\documentclass[12pt,a4paper,twoside]{article}
\usepackage[sectionbib,round]{natbib}
\usepackage{amsmath,amssymb,amsfonts,amsthm,xpatch}
\usepackage{fancyhdr,bm,enumitem,hyperref,graphicx}
\usepackage{multirow,tabularx,authblk,booktabs}
\usepackage{sectsty,secdot,setspace}
\usepackage{algorithm,algorithmic,graphicx}
\usepackage[top = 2.5cm, bottom = 2.5cm, left= 2.5cm, right= 2.5cm]{geometry}
\usepackage[tableposition = top]{caption}
\newcommand{\RNum}[1]{\uppercase\expandafter{\romannumeral #1\relax}}
\newcolumntype{Y}{>{\centering\arraybackslash}X}

\newtheorem{thm}{Theorem}
\newtheorem{lem}{Lemma}
\newtheorem{exm}{Example}
\numberwithin{exm}{section}
\newtheorem{cor}{Corollary}
\newtheorem{defi}{Definition}
\newtheorem{rmk}{Remark}
\newtheorem{prop}{Proposition}
\theoremstyle{definition}
\sectionfont{\fontsize{14pt}{14pt}\selectfont}
\subsectionfont{\fontsize{12pt}{12pt}\selectfont}
\def\spacingset#1{\renewcommand{\baselinestretch}%
{#1}\small\normalsize} \spacingset{1}
\fancypagestyle{plain}{%
\fancyhf{} 

}
\title{\bf \large CONDITIONAL VARIABLE SCREENING VIA ORDINARY LEAST SQUARES PROJECTION}
\author[1]{ Ning Zhang \thanks{Email: nzhang2018@gmail.com}}
\author[2, 1]{ Wenxin Jiang \thanks{Email: wjiang@northwestern.edu}}
\author[3, 2]{ Yuting Lan \thanks{Corresponding author. Email: lan.yuting@mail.shufe.edu.cn}}
\affil[1]{\normalsize School of Mathematics, Shandong University, Jinan, China}
\affil[2]{\normalsize Department of Statistics, Northwestern University, Evanston, IL, U.S.}
\affil[3]{\normalsize Department of Statistics and Management, Shanghai University of Finance and Economics, Shanghai, China}
\date{}
\begin{document}
\maketitle
\vspace{-1.5em}
\noindent\makebox[\textwidth][c]{%
\begin{minipage}{0.9\textwidth}
\spacingset{1.3}
\rule{\textwidth}{0.4pt}
{\bf{Abstract}}\\
In this article, we propose a novel variable screening method for linear models named as conditional screening via ordinary least squares projection (COLP). COLP can take advantage of prior knowledge concerning certain active predictors by eliminating the adverse impact of their coefficients in the estimation of remaining ones and thus significantly enhance the screening accuracy. We prove its sure-screening property under reasonable assumptions and demonstrate its utility in an application to a leukemia dataset. Moreover, based on the conditional approach, we introduce an iterative algorithm named as forward screening via ordinary least squares projection (FOLP), which not only could exploit the prior information more effectively, but also has promising performance when no prior knowledge is available using a data-driven conditioning set. Extensive simulation studies are carried out to demonstrate the competence of both proposed methods.\vspace{2.5pt}
\rule{\textwidth}{0.4pt}
\end{minipage}}
\newpage

\section{Introduction}
\label{sec:intro}

The rapid development of modern information technology during last decades has significantly reduced the cost of data collection and storage. Consequently, scientists nowadays are confronted with unprecedentedly massive data in various scientific fields, such as genomics, economics, signal and image processing and earth sciences, etc. Then how to extract key information from such outsized datasets becomes a great challenge for researchers. Specifically, to identify active predictors (predictors with non-zero coefficients) from the ultrahigh dimensional feature space, statisticians have devoted considerable effort to the research on variable selection and screening techniques.

Recent years have witnessed an explosion in the advancement of variable selection approaches, including but not limited to, the LASSO \citep{LASSO}, the SCAD \citep{SCAD}, the adaptive LASSO \citep{Adalasso} and the elastic net \citep{elastic}.
Through minimizing penalized loss functions, variable selection methods could identify important variables and estimate corresponding parameters simultaneously. Nevertheless, for ultrahigh dimensional data where the predictor dimension expands exponentially with the sample size, many variable selection techniques may no longer be consistent \citep{zhao2006model,Adalasso} and the computational cost of solving high dimensional optimization problems increases dramatically even with the help of some efficient algorithms \citep{LARS,NP}. Concerns on the selection consistency and computational efficiency of variable selection methods motivate the development of variable screening techniques, which are designed to efficiently reduce the predictor dimension to a manageable size such that variable selection approaches can be implemented smoothly afterwards.

\cite{SIS} proposed the seminal sure independence screening (SIS) method to effectively reduce the predictor dimension through ranking marginal correlations between predictors and the response, which is much more efficient in computation compared to solving large-scale optimization problems. Motivated by SIS, a number of variable screening techniques \citep{SISglm,Rcorr,Dcorr} are developed to deal with more general cases applying various marginal utilities. The sure screening property that requires all active predictors to be preserved in the selected model with an overwhelming probability was introduced by \cite{SIS} as a crucial criterion to evaluate the theoretical effectiveness of variable screening methods. \cite{SIS} proved the sure screening property of SIS under the marginal correlation assumption that correlations between active predictors and the response are bounded away from zero, which however, can be easily violated in practice due to high correlations among predictors. Consequently, active predictors that are jointly correlated but marginally uncorrelated with the response are likely to be screened out by SIS, whereas inactive predictors that are highly correlated with active ones have high priority to be selected.

To avoid such undesirable results, \cite{HOLP} introduced another efficient variable screening method named as high dimensional ordinary least squares projection (HOLP), which conducts dimension reduction according to the HOLP estimator constructed by the Moore-Penrose inverse of the design matrix. \cite{HOLP} proved the sure screening property of HOLP without relying on the marginal correlation assumption. Nevertheless, its sure screening property was achieved under the assumption of an upper bound for $||\bm{\beta}||$, the $L_2$ norm of the coefficient vector $\bm{\beta}$. As a result, HOLP may break down when some coefficients are of large absolute values due to their considerable adverse impact on the estimation of other coefficients.

In scientific research, prior information regarding a set of certain active predictors is frequently available from previous studies. For instance, in the analysis of a leukemia dataset, \cite{Golub} found out that two genes, Zyxin and Transcriptional activator hSNF2b, have empirically high correlations with the AML-ALL class distinction. Therefore, further analysis of the leukemia data can be carried out based on this result. To exploit such prior information, \cite{csis} proposed the conditional sure independence screening (CSIS) approach for generalized linear models to identify remaining active predictors through evaluating their conditional contributions to the response conditioning on those known active variables. \cite{csis} proved the sure screening property of CSIS based on the conditional linear covariance assumption, requiring the conditional linear covariances between remaining active predictors and the response to be bounded away from zero. Thus, any active predictor with close-to-zero conditional linear covariance with the response is likely to be screened out by CSIS since its coefficient in the regression with known active predictors is also close to zero.

Motivated by the underperformance of HOLP and CSIS in certain scenarios, we propose a novel conditional variable screening method named as conditional screening via ordinary least squares projection (COLP). COLP initially projects the design matrix onto the orthogonal complement of the column space of conditioning active predictors, and then selects variables according to the estimator for the remaining coefficients constructed by the Moore-Penrose inverse of the projected design matrix. Through the orthogonal projection, COLP could eliminate the negative effect from coefficients of conditioning active predictors in the estimation of remaining ones. The sure screening property of COLP no longer relies on the upper bound of $||\bm{\beta}||$, but only requires the $L_2$ norm of remaining coefficient vector to be bounded from above. Therefore, COLP could identify all remaining active predictors with an overwhelming probability no matter how large the coefficients of conditioning active predictors are. In addition, its sure screening property neither depends on the marginal correlation assumption nor the conditional linear covariance assumption as required by SIS and CSIS, respectively.

We demonstrate the utility of COLP in extensive numerical simulations and an application to a leukemia dataset, where it achieves the best overall performance in the simulation study and yields only one classification error in the analysis of the leukemia data based on only three genes. From the simulation results, we notice that COLP benefits the most from the prior information that covers all significant active predictors (predictors with coefficients of large absolute values). However, it is usually impossible to obtain such informative prior knowledge in scientific research. Therefore, to further enhance the screening accuracy when some significant active predictors are not recognized, we propose an iterative screening method named as forward screening via ordinary least squares projection (FOLP). FOLP employs COLP iteratively conditioning on predictors selected in previous steps, and adds new predictors to the selected model one by one through comparing residual sums of squares (RSS) of candidate models, similar to the classical forward regression \citep[FR]{FR} method. In this way, FOLP is able to diminish the negative impact from the coefficients of unknown active predictors selected in previous iterations. In various simulated scenarios, FOLP exhibits superior screening performance compared to COLP and other commonly used screening techniques. Furthermore, our simulation studies also illustrate that FOLP could work competitively even when no prior information is available by using a data-driven conditioning set selected by HOLP.

The rest of the paper is organized as follows. In Section \ref{sec:colp}, we introduce the COLP method and explain how it could significantly improve the screening accuracy with the help of prior information. In the next, we formally describe the sure screening property of COLP and confirm its numerical effectiveness in three particular examples and an application to a leukemia dataset in Section \ref{sec:perf.colp}. In Section \ref{sec:folp}, we propose the FOLP algorithm and demonstrate its competence in extensive simulation studies. Then we conclude our results and discuss possible future work in Section \ref{sec:conclusion}. Lastly, technical details regarding the sure screening property of COLP can be referred to the appendix.

\section{A new conditional variable screening method: COLP}
\label{sec:colp}
Throughout the paper, we consider the classical linear model
\[y=\bm{x}^\top \bm{\beta}+\epsilon,\]
where $y$ denotes the response, $\bm{x}=(x_1,\cdots,x_p)^\top $ denotes the predictor vector, $\bm{\beta}=(\beta_1,\cdots,\beta_p)^\top $ denotes the coefficient vector and $\epsilon$ denotes the random error. With $n$ realizations of $y$ and $\bm{x}$, we have the alternative model
\[{Y}=X\bm{\beta}+\bm{\epsilon},\]
where ${Y}=(Y_1,\cdots,Y_n)^\top $ denotes the response vector, ${X}\in\mathbb{R}^{n\times p}$ denotes the design matrix and $\bm{\epsilon}=(\epsilon_1,\cdots,\epsilon_n)^\top $ denotes the error vector consisting of $n$ i.i.d random errors. In this paper, we only consider the $p>n$ case. Moreover, let $\mathcal{T}=\{j:\beta_j\neq 0\}$ denote the true model of size $t$, $\mathcal{C}\subsetneqq\mathcal{T}$ denote the conditioning set consisting of indices of $t_c$ active predictors obtained from previous studies, $\mathcal{T}_\mathcal{D}=\mathcal{T}-\mathcal{C}$ denote the remaining true model corresponding to the remaining $t_d = t - t_c$ active predictors and $\mathcal{D}=\{1,\cdots,p\}-\mathcal{C}$ consist of indices of the remaining $p_d = p - t_c$ predictors. Without loss of generality, we further assume that the prior information includes the first $t_c$ predictors, that is, $\mathcal{C}=\{1,\cdots,t_c\}$. For any vector $\bm{v}\in\mathbb{R}^p$ and any index set $\mathcal{S}\subset\{1,\cdots,p\}$, let $\bm{v}_\mathcal{S}$ denote the subvector consisting of the $j$-th entry in $\bm{v}$ with $j\in {\mathcal S}$. Similarly, we denote ${X}_{\mathcal S}$ as the submatrix of ${X}$ with columns corresponding to ${\mathcal S}$.

To identify the remaining active predictors $\bm{x}_{\mathcal{T}_\mathcal{D}}$ based on the prior knowledge of $\bm{x}_\mathcal{C}$, we introduce a new estimator for $\bm{\beta}_{\mathcal{D}}$ as

\begin{equation}\label{est.1}
\hat{\bm{\beta}}_{\mathcal{D}}=(\hat{{\beta}}_{t_c+1},\cdots,\hat{{\beta}}_{p})^\top =({M}_{\mathcal{C}}{X}_{\mathcal{D}})^{+}{Y},
\end{equation}
where ${M}_{\mathcal{C}}$ denotes orthogonal projection matrix on the orthogonal complement of the space spanned by columns of $X_\mathcal{C}$ and $({M}_{\mathcal{C}}{X}_{\mathcal{D}})^{+}$ denotes the Moore-Penrose inverse of ${M}_{\mathcal{C}}{X}_{\mathcal{D}}$. When $X_\mathcal{C}$ is of full column rank, the projection matrix ${M}_{\mathcal{C}}$ can be written as
\[{M}_{\mathcal{C}}=I_n-X_\mathcal{C}(X_\mathcal{C}^\top X_\mathcal{C})^{-1}X_\mathcal{C}^\top,\]
where $I_n$ denotes the $n\times n$ identity matrix.

According to the estimator $\hat{\bm{\beta}}_{\mathcal{D}}$, we then can select the model $\mathcal{S}^\gamma$ applying a threshold parameter $\gamma$ as
\[\mathcal{S}^\gamma=\left\{j\in\mathcal{D}:|\hat\beta_j|>\gamma\right\},\]
or select the model $\mathcal{S}_d$ using a size parameter $d$ as
\[\mathcal{S}_d=\left\{j\in\mathcal{D}:|\hat\beta_j|\text{ are among the largest } d \text{ of all } |\hat\beta_j|\text{s}\right\}.\]

We name the new screening method as conditional screening via ordinary least squares projection (COLP) due to its similarity to the classical ordinary least squares approach and the HOLP method proposed by \cite{HOLP}. In the rest of this section, we will demonstrate how COLP could significantly improve the screening accuracy with the help of the prior information.

\cite{HOLP} introduced the HOLP estimator as
\[\hat{\bm{\beta}}^{*}=(\hat\beta_1^*,\cdots,\hat\beta_p^*)^\top ={X}^\top ({XX}^\top )^{-1}{Y},\]
where the design matrix $X$ is assumed to be of full row rank and ${X}^\top ({XX}^\top )^{-1}=X^+$ in this case. Alternatively, the estimator can be written as
\[\hat{\bm{\beta}}^{*}={X}^{+}Y={X}^{+}{X}\bm{\beta}+{X}^{+}\bm{\epsilon}.\]
Under certain assumptions, \cite{HOLP} proved that the matrix ${X}^{+}{X}$ is diagonally dominant and entries in ${X}^{+}\bm{\epsilon}$ are dominated by corresponding terms in ${X}^{+}{X}\bm{\beta}$ with an overwhelming probability. Consequently, $|\hat{{\beta}}^{*}_i|_{i\in\mathcal{T}}$ can dominate $|\hat{{\beta}}^{*}_j|_{j\not\in\mathcal{T}}$ due to large diagonal terms in ${X}^{+}{X}$. Therefore, the model selected through ranking $|\hat{{\beta}}^{*}_j|$ could preserve all active predictors with an overwhelming probability. However, the screening accuracy of HOLP can be significantly impaired by coefficients of large absolute values in the model, which will be well demonstrated in the following example.

\begin{exm}\label{exm2.1}
We consider the linear model with only two active predictors as
\[y=x_1\beta_1+x_2\beta_2+\sum_{j=3}^px_j\beta_j+\epsilon,\]
where $\beta_1, \beta_2\neq 0$ and $\beta_j=0$ for $3\leq j\leq p$. Notice that even $X^\top ({XX}^\top )^{-1}{X}$ is diagonally dominant under certain assumptions, its off-diagonal terms are usually non-zero. Then if we set the coefficients to $\beta_1=-\frac{m_{22}}{m_{21}}\beta_2$ with $m_{ij}=X_i^\top ({XX}^\top )^{-1}{X}_j$, the HOLP estimator for $\beta_j$ can be written as
\begin{equation}\label{eq.exm21}
\hat\beta_j^*=m_{j1}\beta_1+m_{j2}\beta_2+\hat\epsilon_j=\left(m_{j2}-\frac{m_{22}}{m_{21}}\cdot m_{j1}\right)\beta_2+\hat\epsilon_j,
\end{equation}
where $\hat\epsilon_j=X_j^\top ({XX}^\top )^{-1}\bm{\epsilon}$ for $1\leq j\leq p$.
\end{exm}

Considering $j=2$ in \eqref{eq.exm21}, we can see that the estimate $\hat\beta_2^*$ is only a linear combination of mean-zero random errors, whose absolute value can be easily dominated by those of other estimates. Consequently, the predictor $x_2$ is likely to be screened out by the HOLP method. In this example, HOLP breaks down since $m_{21}\beta_1$ can no longer be dominated by $m_{22}\beta_2$ with $|\beta_1|\gg|\beta_2|$. To ensure the sure screening property of HOLP, \cite{HOLP} set an upper bound for the $L_2$ norm of the coefficient vector through assumptions on the variance of the response and the covariance matrix of predictors.

Nevertheless, significant active predictors are common in real world applications and are more likely to be identified in previous studies. Motivated by the failure of HOLP in such scenarios, we aim to find an efficient screening method that could diminish the adverse impact of non-zero coefficients exploiting the prior information.

With the prior information $\mathcal{C}$, the HOLP estimator for the remaining coefficients $\bm{\beta}_\mathcal{D}$ has the form as
\[\hat{\bm{\beta}}^{*}_{\mathcal{D}}={X}^\top _{\mathcal{D}}({XX}^\top )^{-1}Y={X}^\top _{\mathcal{D}}({XX}^\top )^{-1}{X}_{\mathcal{C}}\bm{\beta}_{\mathcal{C}}+{X}^\top _{\mathcal{D}}({XX}^\top )^{-1}{X}_{\mathcal{D}}\bm{\beta}_{\mathcal{D}}+{X}^\top _{\mathcal{D}}({XX}^\top )^{-1}\bm{\epsilon}.\]
Analogous to Example \ref{exm2.1}, entries in ${X}^\top _{\mathcal{D}}({XX}^\top )^{-1}{X}_{\mathcal{C}}\bm{\beta}_{\mathcal{C}}$ may no longer be dominated by corresponding terms in ${X}^\top _{\mathcal{D}}({XX}^\top )^{-1}{X}_{\mathcal{D}}\bm{\beta}_{\mathcal{D}}$ when $||\bm{\beta}_{\mathcal{C}}||$ is large enough. To diminish the negative impact of $\bm{\beta}_{\mathcal{C}}$ in the estimation of $\bm{\beta}_{\mathcal{D}}$, we propose the COLP estimator for $\bm{\beta}_{\mathcal{D}}$ as
\[\hat{\bm{\beta}}_{\mathcal{D}}=({M}_{\mathcal{C}}{X}_{\mathcal{D}})^{+}{Y}=({M}_{\mathcal{C}}{X}_{\mathcal{D}})^{+}{X}_{\mathcal{C}}\bm{\beta}_{\mathcal{C}}+({M}_{\mathcal{C}}{X}_{\mathcal{D}})^{+}{X}_{\mathcal{D}}\bm{\beta}_{\mathcal{D}}+({M}_{\mathcal{C}}{X}_{\mathcal{D}})^{+}\bm{\epsilon}.\]
According to \cite{mpinv}, the Moore-Penrose inverse $({M}_{\mathcal{C}}{X}_{\mathcal{D}})^{+}$ can be written as
\[{({M}_{\mathcal{C}}{X}_{\mathcal{D}})}^+={X}_{\mathcal{D}}^\top {M}_{\mathcal{C}}{X}_{\mathcal{D}}({X}_{\mathcal{D}}^\top {M}_{\mathcal{C}}{X}_{\mathcal{D}}{X}_{\mathcal{D}}^\top {M}_{\mathcal{C}}{X}_{\mathcal{D}})^-{X}_{\mathcal{D}}^\top {M}_{\mathcal{C}},\]
where $({X}_{\mathcal{D}}^\top {M}_{\mathcal{C}}{X}_{\mathcal{D}}{X}_{\mathcal{D}}^\top {M}_{\mathcal{C}}{X}_{\mathcal{D}})^-$ denotes the corresponding generalized inverse. As a result, we have $({M}_{\mathcal{C}}{X}_{\mathcal{D}})^+=({M}_{\mathcal{C}}{X}_{\mathcal{D}})^+{M}_{\mathcal{C}}$ since ${M}_{\mathcal{C}}^2={M}_{\mathcal{C}}$. Therefore, the COLP estimator can be further expressed as
\begin{align*}
\hat{\bm{\beta}}_{\mathcal{D}}&=({M}_{\mathcal{C}}{X}_{\mathcal{D}})^{+}{M}_{\mathcal{C}}{X}_{\mathcal{C}}\bm{\beta}_{\mathcal{C}}+({M}_{\mathcal{C}}{X}_{\mathcal{D}})^{+}{M}_{\mathcal{C}}{X}_{\mathcal{D}}\bm{\beta}_{\mathcal{D}}+({M}_{\mathcal{C}}{X}_{\mathcal{D}})^{+}\bm{\epsilon}\\
&=({M}_{\mathcal{C}}{X}_{\mathcal{D}})^{+}{M}_{\mathcal{C}}{X}_{\mathcal{D}}\bm{\beta}_{\mathcal{D}}+({M}_{\mathcal{C}}{X}_{\mathcal{D}})^{+}\bm{\epsilon},
\end{align*}
where the last equation comes from the fact that ${M}_{\mathcal{C}}{X}_{\mathcal{C}}=0$. In this way, the COLP estimator could eliminate the influence of $\bm{\beta}_{\mathcal{C}}$ in the estimation of $\bm{\beta}_{\mathcal{D}}$. Furthermore, as shown in the supplementary material, $({M}_{\mathcal{C}}{X}_{\mathcal{D}})^{+}{M}_{\mathcal{C}}{X}_{\mathcal{D}}$ is a diagonally dominant matrix and entries in $({M}_{\mathcal{C}}{X}_{\mathcal{D}})^{+}\bm{\epsilon}$ can be dominated by corresponding terms of $({M}_{\mathcal{C}}{X}_{\mathcal{D}})^{+}{M}_{\mathcal{C}}{X}_{\mathcal{D}}\bm{\beta}_{\mathcal{D}}$ under proper assumptions. Consequently, the model selected through ranking absolute values of COLP estimates could preserve all the remaining active predictors with an overwhelming probability without imposing any restriction on $\bm{\beta}_{\mathcal C}$. At the end, it is also noteworthy that when the prior information is not available (i.e., $\mathcal{C}=\emptyset$), we have $X_\mathcal{D}=X$, $\bm{\beta}_\mathcal{D}=\bm{\beta}$ and ${M}_\mathcal{C}={I}_n$, and thus the COLP estimator degenerates to the HOLP one.

\section{Theoretical and numerical performance of COLP}
\label{sec:perf.colp}
In this section, we introduce the sure screening property of COLP and examine its numerical utility in three
simulated cases and an application to a leukemia dataset.
\subsection{The sure screening property of COLP}
The sure screening property of COLP relies on the following three assumptions.
\begin{enumerate}[nosep,label=({A\arabic*}),align=left, leftmargin=*]
    \item The predictor vector $\bm{x}$ follows a centered multivariate normal distribution with a covariance matrix as
    \[\Sigma=\begin{bmatrix}\Sigma_{11}&\Sigma_{12}\\\Sigma_{21}&\Sigma_{22}\end{bmatrix},\]
    where $\Sigma_{11}$ and $\Sigma_{22}$ are $t_c\times t_c$ and $p_d\times p_d$ matrices with unit diagonal elements.
    \item The random error $\epsilon$ is independent of $\bm{x}$ and follows a centered sub-Gaussian distribution with a variance of $\sigma^2$.
    \item $\mathrm{E}[\mathrm{var}(y|\bm{x}_{\mathcal{C}})]=O(1)$ and there exist constants $c_t,c_\beta> 0$, $c_\lambda\geq 1$ and $\xi_t, \xi_\beta, \xi_\lambda\geq 0$ with $\xi_t+2\xi_\beta+5\xi_\lambda<1$, such that
    \[t\leq c_t n^{\xi_t},\,\,\beta_{\text{min}}=\min_{j\in\mathcal{T}}|\beta_j|\geq c_\beta n^{-\xi_\beta}\,\, \text{and}\,\,\text{cond}(\Sigma)=\frac{\lambda_{\text{max}}(\Sigma)}{\lambda_{\text{min}}(\Sigma)}\leq c_\lambda n^{\xi_\lambda},\]
    where $\lambda_{\text{max}}(\Sigma)$ and $\lambda_{\text{min}}(\Sigma)$ denote the largest and smallest eigenvalues of $\Sigma$, respectively.
\end{enumerate}

The normality assumption (A1) was assumed by \cite{FR} to facilitate the proof of the sure screening property of FR and can be regarded as a special case of assumptions on the distribution of $\bm{x}$ made by \cite{SIS} and \cite{HOLP}.

For simplicity, we only consider sub-Gaussian distributed random errors in assumption (A2), including normal distributed, Bernoulli distributed and other bounded errors. In fact, the sure screening property of COLP still holds for random errors from sub-exponential distributions or distributions with bounded $2k$-th moments since they have similar tail behaviors with sub-Gaussian distributions according to \cite{ver} and \cite{HOLP}.

Moreover, in assumption (A3), we adopt the same restrictions as assumed in \cite{HOLP} except that we set a constant upper bound for the expected conditional variance $\text{E}[\mathrm{var}(y|\bm{x}_{\mathcal{C}})]$ instead of the variance $\text{var}(y)$, which is a weaker condition according to the law of total variance
\[\text{var}(y)=\text{E}[\mathrm{var}(y|\bm{x}_{\mathcal{C}})]+\text{var}[\mathrm{E}(y|\bm{x}_{\mathcal{C}})].\]
Notice that if $\bm{x}$ follows a centered multivariate normal distribution with the covariance matrix satisfying $\text{cond}(\Sigma)\leq c_\lambda n^{\xi_\lambda}$,
a constant upper bound for $\text{E}[\mathrm{var}(y|\bm{x}_{\mathcal{C}})]$ only implies that $||\bm{\beta}_\mathcal{D}||^2\leq C_\beta n^{\xi_\lambda}$ for some positive constant $C_\beta$, whereas $\text{var}(y)=O(1)$ indicates that $||\bm{\beta}||^2\leq C^*_\beta n^{\xi_\lambda}$ for some $C^*_\beta>0$. Therefore, the sure screening property of COLP remains valid with $||\bm{\beta}_\mathcal{C}||$ being arbitrarily large, while that of HOLP may be violated due to the large $L_2$ norm of $\bm{\beta}$.

Additionally, the condition $t=o(n)$ implies that only a small portion of predictors are relevant to the response. Therefore, with $n_d=n-t_c$, all our theoretical results that involve $n_d$ and $p_d$ can be expressed in terms of $n$ and $p$ since $n_d=O(n)$ and $p_d=O(p)$. Finally, it is also noteworthy that the sure screening property of COLP does not depend on the marginal correlation assumption necessary for SIS nor the conditional linear covariance assumption required by CSIS.

The sure screening property of COLP can be formally described in the following two theorems, corresponding to models selected using a threshold parameter $\gamma_n$ and a size parameter $d_n$, respectively.
\begin{thm}\label{thm1}
Under assumptions (A1), (A2) and (A3), if we select the model $\mathcal{S}^{\gamma_n}=\{j\in\mathcal{D}:|\hat\beta_j|>\gamma_n\}$ with a threshold parameter ${\gamma_n}$ satisfying that
\[\frac{n^{1-\xi_\beta-\xi_\lambda}}{\sqrt{\log n}\cdot p}=o(\gamma_n)\quad\text{and}\quad \gamma_n = o\left(\frac{n^{1-\xi_\beta-\xi_\lambda}}{p}\right),\]
then there exists some constant $C>0$, such that
\[P\left(\mathcal{T}_\mathcal{D}\subset \mathcal{S}^{\gamma_n}\right)\geq 1- O\left\{\exp\left(\frac{-C\cdot n^{1-\xi_t-2\xi_\beta-5\xi_\lambda}}{{\log n}}\right)\right\}.\]
\end{thm}

Theorem \ref{thm1} indicates that the model selected by COLP using an appropriate threshold parameter could preserve all the remaining active predictors with an overwhelming probability under assumptions (A1)-(A3). Furthermore, if we select the model of $d_n$ predictors according to the absolute values of corresponding COLP estimates, all the remaining active predictors can also be identified with an overwhelming probability under an additional assumption on the predictor dimension $p$.

\begin{thm}\label{thm2}
Suppose assumptions (A1), (A2) and (A3) hold and $p$ satisfies that
\[\log p = o\left(\frac{n^{1-\xi_t-2\xi_\beta-5\xi_\lambda}}{{\log n}}\right).\]
Then if we select the model $\mathcal{S}_{d_n}=\{j\in\mathcal{D}:|\hat\beta_j|\text{ are among the largest } d_n \text{ of all } \}$ with $d_n\geq c_t n^{\xi_t}$, there exists the same constant $C$ as chosen in Theorem \ref{thm1}, such that
\[P\left(\mathcal{T}_\mathcal{D}\subset \mathcal{S}_{d_n}\right)\geq 1- O\left\{\exp\left(\frac{-C\cdot n^{1-\xi_t-2\xi_\beta-5\xi_\lambda}}{{\log n}}\right)\right\}.\]
\end{thm}

Theorem \ref{thm1} and \ref{thm2} together provide theoretical guarantee for the COLP approach. In the rest of this section, we will further demonstrate the attractiveness of COLP in the simulation study and an analysis of a leukemia dataset.

\subsection{Simulation study \RNum{1}: conditional screening performance of COLP}\label{sim1}

In this section, we evaluate the conditional screening performances of COLP in three different scenarios, where SIS, CSIS and HOLP may fail to identify all the remaining active predictors due to violations of certain assumptions that their sure screening properties rely on.

In each case, we compare the screening performances of these four methods applying two parameter setups, $(d,n,p)=(200,100,$ $2000)$ and $(d,n,p)=(100,200,10000)$, where $d$ denotes the number of simulations. Additionally, the prior information $\mathcal{C}$ is chosen as $\{1\}$, $\{2\}$ or $\{3,4\}$ in all the examples. With the prior information, SIS and HOLP determine the selected models according to the estimates for the remaining coefficients $\bm{\beta}_{\mathcal{D}}$. Moreover, the size of selected models is set to $d_n=\lfloor n/ \log(n)\rfloor$ and the random error $\epsilon$ is assumed to follow a centered normal distribution with the variance $\sigma^2$ adjusted to achieve $R^2 = 60 \%$ or $R^2 = 90 \%$, where $R^2 =\text{var}(\bm{x}^T\bm{\beta})/\text{var}(y)$ denotes the signal ratio defined by \cite{FR}. Finally, the conditional screening performance is evaluated based on the following criteria.
\begin{itemize}[align=left, leftmargin=*]
\item $\mathbb{P}_s$: the proportion of simulations where all remaining active predictors are included in the selected model.

\item $\mathbb{M}_s$: the median of minimum model sizes (MMS) of selected models that are required to cover all remaining active predictors. The sampling variability of MMS is measured by the robust standard deviation (RSD), which is defined as the associated interquartile range of MMS divided by 1.34.
\end{itemize}

\begin{exm}\label{exm3.1}
We consider the linear model
\[y=5x_1+x_2+x_3+x_4+\sum\limits_{j=5}^px_j\beta_j+\epsilon,\]
where $\bm{x}$ follows the multivariate normal distribution $ N(\bm{0},I_p)$ and $\beta_j=0$ for $j\geq 5$.
\end{exm}

\begin{exm}\label{exm3.2}
We have the linear model
\[y=5x_1+2x_2+2x_3+2x_4-4x_5+\sum\limits_{j=6}^px_j\beta_j+\epsilon,\]
where $\bm{x}$ follows a centered multivariate normal distribution with $\mathrm{cov}(x_1,x_j)=0.5$ for $j\neq 1$ and $\mathrm{cov}(x_i,x_j)=0.75$ for $i,j>1$ and $i\neq j$. The coefficients $\beta_j$s are set to $0$ for $j\geq 6$.
\end{exm}

\begin{exm}\label{exm3.3}
We investigate the linear model
\[y=5x_1+x_2+2x_3+2x_4+2x_5-3x_6+\sum\limits_{j=7}^px_j\beta_j+\epsilon,\]
where $\bm{x}$ follows a centered multivariate normal distribution with $\mathrm{cov}(x_1,x_j)=0$ for $j\neq 1$, $\mathrm{cov}(x_2,x_j)=0$ for $j\neq 2$ and $\mathrm{cov}(x_i,x_j)=0.5$ for $i,j>2$ and $i\neq j$. Additionally, we set $\beta_j=0$ for $j\geq 7$.
\end{exm}

\begin{table}[!htb]
\scriptsize
\renewcommand{\arraystretch}{1}
\caption{\footnotesize The proportion of sure screening $\mathbb{P}_s$, the median of minimum required model sizes $\mathbb{M}_s$ and its robust standard deviation RSD (in parentheses) for screening methods in Example \ref{exm3.1}.}\label{tab3.1}
\begin{tabularx}{\textwidth}{@{}c c Y r Y r Y r Y r@{}}
\toprule
\multirow{2}{*}[-0.4em]{$\mathcal{C}$}
&\multirow{2}{*}[-0.4em]{$R^2$}
&\multicolumn{2}{c}{SIS}
&\multicolumn{2}{c}{CSIS}
&\multicolumn{2}{c}{HOLP}
&\multicolumn{2}{c}{COLP} \\
\cmidrule(l){3-4} \cmidrule(l){5-6} \cmidrule(l){7-8} \cmidrule(l){9-10}
& & $\mathbb{P}_s$ & $\mathbb{M}_s$(RSD)
& $\mathbb{P}_s$ & $\mathbb{M}_s$(RSD)
& $\mathbb{P}_s$ & $\mathbb{M}_s$(RSD)
& $\mathbb{P}_s$ & $\mathbb{M}_s$(RSD) \\
\midrule
\multicolumn{10}{c}{$(d,n,p)=(200,100,2000)$}\\
\midrule
$\{1\}$ & $ 60\%$& 0.00 & 902 (722)& 0.04 & 364 (549)& 0.00 & 878 (732)& 0.02 & 397 (558)\\
& $ 90\%$& 0.01 & 597 (616)& 0.87 & 5 (5)& 0.01 & 598 (532)& 0.85 & 5 (6)\\
\midrule
$\{2\}$ & $ 60\%$& 0.03 & 645 (755)& 0.03 & 653 (764)& 0.04 & 618 (790)& 0.04 & 626 (771)\\
& $ 90\%$& 0.04 & 452 (637)& 0.04 & 446 (626)& 0.03 & 452 (499)& 0.04 & 439 (549)\\
\midrule
$\{3, 4\}$ & $ 60\%$& 0.12 & 241 (459)& 0.12 & 254 (493)& 0.12 & 230 (436)& 0.12 & 238 (473)\\
& $ 90\%$& 0.21 & 96 (278)& 0.26 & 117 (289)& 0.24 & 93 (272)& 0.29 & 103 (251)\\
\midrule
\multicolumn{10}{c}{$(d,n,p)=(100,200,10000)$}\\
\midrule
$\{1\}$ & $ 60\%$& 0.00 & 2420 (2777)& 0.08 & 426 (844)& 0.00 & 2290 (2633)& 0.10 & 412 (691)\\
& $ 90\%$& 0.09 & 644 (1415)& 1.00 & 3 (0)& 0.07 & 582 (1475)& 1.00 & 3 (0)\\
\midrule
$\{2\}$ & $ 60\%$& 0.02 & 1378 (2067)& 0.02 & 1248 (2182)& 0.02 & 1333 (2010)& 0.02 & 1200 (2141)\\
& $ 90\%$& 0.14 & 302 (895)& 0.18 & 296 (762)& 0.13 & 248 (1005)& 0.19 & 280 (761)\\
\midrule
$\{3, 4\}$ & $ 60\%$& 0.22 & 509 (1285)& 0.23 & 372 (1093)& 0.23 & 466 (1251)& 0.25 & 442 (1027)\\
& $ 90\%$& 0.46 & 53 (235)& 0.52 & 32 (154)& 0.44 & 52 (236)& 0.50 & 38 (146)\\
\bottomrule
\end{tabularx}
\end{table}

\begin{table}[!htb]
\scriptsize
\renewcommand{\arraystretch}{1}
\caption{\footnotesize The proportion of sure screening $\mathbb{P}_s$, the median of minimum required model sizes $\mathbb{M}_s$ and its robust standard deviation RSD (in parentheses) for screening methods in Example \ref{exm3.2}.}\label{tab3.2}
\begin{tabularx}{\textwidth}{@{}c c Y r Y r Y r Y r@{}}
\toprule
\multirow{2}{*}[-0.4em]{$\mathcal{C}$}
&\multirow{2}{*}[-0.4em]{$R^2$}
&\multicolumn{2}{c}{SIS}
&\multicolumn{2}{c}{CSIS}
&\multicolumn{2}{c}{HOLP}
&\multicolumn{2}{c}{COLP} \\
\cmidrule(l){3-4} \cmidrule(l){5-6} \cmidrule(l){7-8} \cmidrule(l){9-10}
& & $\mathbb{P}_s$ & $\mathbb{M}_s$(RSD)
& $\mathbb{P}_s$ & $\mathbb{M}_s$(RSD)
& $\mathbb{P}_s$ & $\mathbb{M}_s$(RSD)
& $\mathbb{P}_s$ & $\mathbb{M}_s$(RSD) \\
\midrule
\multicolumn{10}{c}{$(d,n,p)=(200,100,2000)$}\\
\midrule
$\{1\}$& $60\%$& 0.00 & 1988 (43)& 0.00 & 1971 (168)& 0.00 & 1118 (684)& 0.00 & 924 (779)\\
& $90\%$& 0.00 & 1998 (4)& 0.00 & 1999 (1)& 0.01 & 478 (612)& 0.21 & 83 (151)\\
\midrule
$\{2\}$& $60\%$& 0.00 & 1988 (43)& 0.01 & 1617 (721)& 0.01 & 916 (760)& 0.01 & 956 (750)\\
& $90\%$& 0.00 & 1998 (4)& 0.00 & 1568 (760)& 0.06 & 281 (576)& 0.05 & 266 (604)\\
\midrule
$\{3, 4\}$& $60\%$& 0.00 & 1988 (43)& 0.01 & 517 (673)& 0.04 & 384 (658)& 0.04 & 399 (653)\\
& $90\%$& 0.00 & 1997 (4)& 0.16 & 183 (310)& 0.27 & 118 (248)& 0.27 & 98 (240)\\
\midrule
\multicolumn{10}{c}{$(d,n,p)=(100,200,10000)$}\\
\midrule
$\{1\}$& $60\%$& 0.00 & 9997 (8)& 0.00 & 9997 (74)& 0.01 & 2653 (3334)& 0.01 & 1406 (2442)\\
& $90\%$& 0.00 & 9999 (0)& 0.00 & 9999 (0)& 0.05 & 454 (1017)& 0.59 & 18 (43)\\
\midrule
$\{2\}$& $60\%$& 0.00 & 9997 (8)& 0.01 & 8908 (3501)& 0.03 & 1323 (2674)& 0.02 & 1277 (2536)\\
& $90\%$& 0.00 & 9999 (0)& 0.00 & 8045 (3911)& 0.09 & 265 (715)& 0.10 & 234 (705)\\
\midrule
$\{3, 4\}$& $60\%$& 0.00 & 9996 (8)& 0.06 & 1005 (1923)& 0.07 & 700 (2122)& 0.09 & 646 (1994)\\
& $90\%$& 0.00 & 9998 (0)& 0.46 & 53 (208)& 0.52 & 31 (192)& 0.54 & 30 (184)\\
\bottomrule
\end{tabularx}
\end{table}

\begin{table}[!htb]
\scriptsize
\renewcommand{\arraystretch}{1}
\caption{\footnotesize The proportion of sure screening $\mathbb{P}_s$, the median of minimum required model sizes $\mathbb{M}_s$ and its robust standard deviation RSD (in parentheses) for screening methods in Example \ref{exm3.3}.}\label{tab3.3}
\begin{tabularx}{\textwidth}{@{}c c Y r Y r Y r Y r@{}}
\toprule
\multirow{2}{*}[-0.4em]{$\mathcal{C}$}
&\multirow{2}{*}[-0.4em]{$R^2$}
&\multicolumn{2}{c}{SIS}
&\multicolumn{2}{c}{CSIS}
&\multicolumn{2}{c}{HOLP}
&\multicolumn{2}{c}{COLP} \\
\cmidrule(l){3-4} \cmidrule(l){5-6} \cmidrule(l){7-8} \cmidrule(l){9-10}
& & $\mathbb{P}_s$ & $\mathbb{M}_s$(RSD)
& $\mathbb{P}_s$ & $\mathbb{M}_s$(RSD)
& $\mathbb{P}_s$ & $\mathbb{M}_s$(RSD)
& $\mathbb{P}_s$ & $\mathbb{M}_s$(RSD) \\
\midrule
\multicolumn{10}{c}{$(d,n,p)=(200,100,2000)$}\\
\midrule
$\{1\}$& $60\%$& 0.00 & 1939 (147)& 0.00 & 1986 (84)& 0.01 & 948 (737)& 0.01 & 581 (576)\\
& $90\%$& 0.00 & 1982 (128)& 0.00 & 1999 (2)& 0.01 & 494 (657)& 0.47 & 24 (44)\\
\midrule
$\{2\}$& $60\%$& 0.00 & 1909 (241)& 0.00 & 1918 (237)& 0.01 & 718 (718)& 0.01 & 734 (720)\\
& $90\%$& 0.00 & 1972 (156)& 0.00 & 1972 (141)& 0.03 & 224 (373)& 0.03 & 232 (365)\\
\midrule
$\{3, 4\}$& $60\%$& 0.00 & 1938 (150)& 0.00 & 1106 (721)& 0.03 & 646 (771)& 0.03 & 637 (765)\\
& $90\%$& 0.00 & 1980 (128)& 0.01 & 644 (748)& 0.13 & 189 (477)& 0.17 & 214 (397)\\
\midrule
\multicolumn{10}{c}{$(d,n,p)=(100,200,10000)$}\\
\midrule
$\{1\}$& $60\%$& 0.00 & 9966 (208)& 0.00 & 9995 (22)& 0.01 & 2058 (2710)& 0.04 & 612 (885)\\
& $90\%$& 0.00 & 9994 (47)& 0.00 & 9999 (0)& 0.12 & 291 (838)& 0.86 & 6 (8)\\
\midrule
$\{2\}$& $60\%$& 0.00 & 9962 (363)& 0.00 & 9969 (266)& 0.04 & 1371 (2073)& 0.03 & 1424 (1780)\\
& $90\%$& 0.00 & 9994 (50)& 0.00 & 9993 (34)& 0.20 & 154 (302)& 0.22 & 156 (297)\\
\midrule
$\{3, 4\}$& $60\%$& 0.01 & 9966 (208)& 0.00 & 2785 (3453)& 0.09 & 1332 (1877)& 0.09 & 1080 (1653)\\
& $90\%$& 0.00 & 9993 (47)& 0.04 & 1208 (2801)& 0.34 & 98 (443)& 0.45 & 67 (363)\\
\bottomrule
\end{tabularx}
\end{table}

Simulation results are presented in Table \ref{tab3.1}-\ref{tab3.3}, from which we can summarize that COLP has the overall best screening performance, whereas SIS fails to identify all the remaining active predictors in most simulations and CSIS and HOLP break down in various scenarios.

Specifically, in Example \ref{exm3.1}, we consider the linear model with independent predictors and the active predictor $x_1$ has a relatively large coefficient. From Table \ref{tab3.1}, we can see that all these four methods have similar performance when $R^2=60\%$ as a consequence of large random errors. However, when the signal ratio is large and the prior information $\mathcal{C}=\{1\}$ is available, SIS still fails to identify all the remaining active predictors because of small marginal correlations between remaining active predictors and the response. Moreover, HOLP also breaks down in this case as $\beta_1$ has substantial negative impact on the estimation of other coefficients. On the contrary, COLP could eliminate the influence of $\beta_1$ in the estimation and identify all the remaining active predictors in $85\%$ of simulations when $p=2000$ and in all the simulations when $p=10000$. With the help of the prior knowledge, CSIS has similar performance with COLP in this example.

In Example \ref{exm3.2}, we consider the linear model with highly correlated predictors and the active predictor $x_5$ is designed to have zero conditional linear covariance with $y$ conditioning on $x_1$, which is computed as
\[\mathrm{cov}_L(x_5,y|x_1)=\sum_{i=1}^p(\sigma_{i5}-\sigma_{15}\sigma_{11}^{-1}\sigma_{1i})\beta_i=0,\]
where $\Sigma=[\sigma_{ij}]$ denotes the covariance matrix of $\bm{x}$. Recall that the sure screening property of CSIS requires the conditional linear covariances between remaining active predictors and the response to be bounded away from zero. Consequently, we see that CSIS breaks down with the prior information $\mathcal{C}=\{1\}$ under this setting. Moreover, due to high correlations between inactive and active predictors, SIS also fails to identify all the remaining active predictors. As shown in Table \ref{tab3.2}, $\mathbb{M}_s\approx p_d$ for CSIS and SIS when $\mathcal{C}=\{1\}$ and $\mathcal{C}=\{2\}$, indicating that they need to select almost all the remaining predictors to cover the true model. Similarly, the screening performance of HOLP is also significantly impaired by the large coefficient $\beta_1$ in this example. Meanwhile, COLP could overcome all these difficulties and enhance the screening accuracy exploiting the prior information, especially in the $\mathcal{C}=\{1\}$ case.

Finally, we consider a more challenging case in Example \ref{exm3.3}, where the active predictor $x_6$ has zero marginal correlation as well as zero conditional linear covariance with the response conditioning on $x_1$ or $x_2$, that is,
\[\mathrm{cov}(x_6,y)=0,\,\,\mathrm{cov}_L(x_6,y|x_1)=0\,\,\text{and}\,\,\mathrm{cov}_L(x_6,y|x_2)=0.\]
Violation of the marginal correlation assumption and the conditional linear covariance assumption results in the underperformance of SIS and CSIS in this case. As presented in Table \ref{tab3.3}, they have to select almost all the remaining predictors to include the active ones when $\mathcal{C}=\{1\}$ and $\mathcal{C}=\{2\}$. Since the sure screening property of HOLP does not depend on those covariance assumptions, it has slightly better performance in this example. Similarly, COLP could further improve the screening accuracy by eliminating the impact from coefficients of known active predictors.

From above discussions, we can conclude COLP as a competitive screening method compared to other three commonly used techniques, which could avoid the restrictions of the marginal correlation assumption and the conditional linear covariance assumption and be immune from the negative influence of coefficients of active predictors in the prior information. In the next, we will further illustrate the effectiveness of COLP in an analysis of a leukemia dataset.

\subsection{Real data analysis: a leukemia dataset}
In this section, we demonstrate how COLP could be applied to conduct variable selection in the analysis of a leukemia dataset that was first studied by \cite{Golub}, which investigated the gene expression of 7129 genes in two types of acute leukemias, acute lymphoblastic leukemia (ALL) and acute myeloid leukemia (AML). The dataset includes 72 samples (47 ALL and 25 AML), where 38 (27 ALL and 11 AML) of them are designed as training samples and the rest 34 (20 ALL and 14 AML) are chosen as testing samples. Compared to ALL, AML progresses rapidly and has a poor prognosis. Therefore, how to make consistent classification of ALL and AML based on expression of selected genes is crucial for the diagnosis.

In the study of \cite{Golub}, two genes, Zyxin and Transcriptional activator hSNF2b, were identified to have empirically high correlations with the difference between ALL and AML patients. Taking these two genes as the prior information, we identify another gene Myeloperoxidase (MPO) applying the COLP method. The expression of MPO is widely accepted as a golden marker for the diagnosis of AML and its prognostic significance in AML is demonstrated in various literatures \citep{MPO.1,MPO.2,MPO.3}. Based on the three genes, we then employ the logistic regression method (COLP-LR) or the naive Bayes rule (COLP-NB) to conduct the final classification.

In Table \ref{tableuk}, we present the classification results of COLP-LR, COLP-NB and other commonly used techniques, including CSIS, NSC \citep[nearest shrunken centroids]{NSC}, SIS-SCAD-LD and SIS-SCAD-NB, which denote the methods that employ the linear discrimination rule and the naive Bayes rule on the model selected by the SIS and SCAD techniques, respectively. The results of CSIS, NSC and SIS are extracted directly from \cite{csis}, \cite{NSC} and \cite{SIS}. From Table \ref{tableuk}, we can see that COLP-LR, COLP-NB and CSIS all have promising results as they achieve a training error of 0 out of 38 and a test error of 1 out of 34 based on only three genes, whereas NSC and SIS require ten more genes to make the classification and yield less accurate results.

\begin{table}[!htb]
\centering
\small
\renewcommand{\arraystretch}{1.2}
\caption{\footnotesize Classification results of various methods for the leukemia dataset.}\label{tableuk}
\begin{tabularx}{0.8\textwidth}{c Y c Y}
\toprule
Method & Training error & Testing error & Number of genes \\
\midrule
COLP-LR & 0/38 & 1/34 & 3 \\
COLP-NB & 0/38 & 1/34 & 3 \\
CSIS & 0/38 & 1/34 & 3 \\
NSC & 1/38 & 2/34 & 21 \\
SIS-SCAD-LD & 0/38 & 1/34 & 16 \\
SIS-SCAD-NB & 4/38 & 1/34 & 16 \\
\bottomrule
\end{tabularx}
\end{table}

\section{An extension of the COLP method}
\label{sec:folp}
In the previous section, we introduce the asymptotic property of COLP and verify its numerical competence in three simulated scenarios and an application to the leukemia dataset. From the simulation results, we notice that COLP achieves the best performance when the prior information includes all the significant active predictors. Nevertheless, it is usually unrealistic to obtain such informative prior knowledge in scientific researches. To further improve the screening accuracy when some significant active predictors are left out of the prior information, we propose an iterative algorithm in this section as an extension of COLP that could iteratively eliminate the possible impact from coefficients of predictors selected in previous steps and then demonstrate the advantage of the iterative approach in extensive numerical studies.


\subsection{Forward screening via ordinary least squares projection}
\label{subsec:folp}
In the simulation study \RNum{1}, we notice that the screening performance of COLP conditioning on $\mathcal{C}=\{1\}$ far exceeds that conditioning on $\mathcal{C}=\{2\}$ and $\mathcal{C}=\{3,4\}$. Such an obvious difference is caused by the large coefficient of $x_1$ and relatively small coefficients of other active predictors in all  three scenarios. On one hand, small coefficients have slight effect on the estimation of remaining ones and thus eliminating their impact will not bring any substantial improvement. On the other hand, with $x_1$ outside the prior information, its large coefficient can still  significantly impair the estimation of remaining parameters.

In practice, researchers usually have to conduct variable selection with some significant active predictors unidentified. To better deal with such common situations, we seek a new screening method that could diminish the impact from coefficients of those hidden significant active predictors. A natural solution is to apply COLP iteratively to eliminate the negative influence from possible large coefficients of predictors selected in previous steps. Inspired by the FR method \citep{FR}, we propose the following iterative algorithm.
\begin{enumerate}[align=left, leftmargin=*]
\item (Initialization) Apply COLP conditioning on the prior information $\mathcal{C}$ and obtain the index set $\mathcal{D}^*$ as a permutation of $\mathcal{D}$ through ranking absolute values of the COLP estimates $\{\hat\beta_j\}_{j\in\mathcal{D}}$ in a decreasing order. Denote $\mathcal{D}^*[1]$ as the first element in $\mathcal{D}^*$ and $\mathcal{D}^*[-1]$ as the subset of $\mathcal{D}^*$ without its first element. Then set $\mathcal{S}_1=\mathcal{D}^*[1]$ and $\mathcal{A}_1=\mathcal{D}^*[-1]$.

\item (Iteration) In the $i$-th iteration, based on the new conditioning set $\mathcal{C}_i=\mathcal{C}\cup\mathcal{S}_i$, we obtain the COLP estimator $\hat{\bm{\beta}}_{\mathcal{D}_i}$ for the remaining coefficients with $\mathcal{D}_i=\{1,\cdots,p\}-\mathcal{C}_i$ and achieve the corresponding permuted index set $\mathcal{D}_i^*$ as in Step 1. Consider two candidate models $\mathcal{M}_{i1}=\mathcal{C}_i\cup \mathcal{A}_i[1]$ and $\mathcal{M}_{i2}=\mathcal{C}_i\cup \mathcal{D}^*_i[1]$ and compute corresponding RSS as
\[\mathrm{RSS}_{ik}=Y^\top\left[{I}_n - {X}_{\mathcal{M}_{ik}}({X}^\top_{\mathcal{M}_{ik}}{X}_{\mathcal{M}_{ik}})^{-1}{X}^\top_{\mathcal{M}_{ik}}\right]Y,\quad k=1\text{ and }2.\]
If $\mathrm{RSS}_{i1}\leq \mathrm{RSS}_{i2}$, set $\mathcal{S}_{i+1}=\mathcal{S}_i\cup\mathcal{A}_i[1]$ and $\mathcal{A}_{i+1}=\mathcal{A}_i[-1]$. Otherwise, set $\mathcal{S}_{i+1}=\mathcal{S}_i\cup\mathcal{D}^*_i[1]$ and $\mathcal{A}_{i+1}=\mathcal{D}^*_i[-1]$.

\item (Solution path) Iterate Step 2 until we get the model $\mathcal{S}_{d_n}$ of size $d_n$ for some predetermined size parameter $d_n$. At the same time, we also obtain a collection of nested models $\mathbb{S}_{d_n}=\{\mathcal{S}_1,\cdots,\mathcal{S}_{d_n}\}$ and name it as the solution path of the algorithm.
\end{enumerate}

We name the iterative algorithm as forward screening via ordinary least squares projection (FOLP) since it applies COLP iteratively and selects predictors one by one through comparing RSS of candidate models similar to FR. However, in each iteration, FR computes the RSS for all the remaining predictors, whereas FOLP only evaluates the RSS of two candidate models, which could considerably drop the computational cost in the high dimensional scenario. Moreover, FR fails to select models of size larger than $n$ due to the limitation of degrees of freedom. In contrast, we can always add new predictors to the selected model according to $\mathcal{A}_n$ once we obtained the model $\mathcal{S}_n$ of size $n$ in the FOLP algorithm.

Furthermore, compared to algorithms that simply apply COLP for several times, FOLP could save us from the additional parameter tuning procedure of deciding the number of times to employ COLP and the number of predictors to select after each time COLP is applied. Both tuning parameters can be determined by comparing the RSS of candidate models in the FOLP algorithm. For instance, when $\mathrm{RSS}_{i1}\leq \mathrm{RSS}_{i2}$ for $1\leq i \leq d_n$, FOLP can be regarded as COLP with the prior information $\mathcal{C}$. On the contrary, if $\mathrm{RSS}_{i1}> \mathrm{RSS}_{i2}$ for $1\leq i \leq d_n$, FOLP is equivalent to employing COLP for $d_n$ times and selecting only one predictor after each time COLP is applied. Lastly, it is noteworthy that FOLP can work competitively without any prior information. As introduced in Section \ref{sec:colp}, with the prior information $\mathcal{C}=\emptyset$, COLP degenerates to the HOLP method. Therefore, in this case, we can choose the prior information as the first predictor selected by HOLP and then employ the FOLP algorithm based on this data-driven conditioning set.

\subsection{Simulation study \RNum{2}: a revisit of Simulation study \RNum{1}}
In this section, we compare the conditional screening performance of COLP and FOLP in Example \ref{exm3.1}-\ref{exm3.3} under the same settings as in Simulation study \RNum{1}. For simplicity, we omit the performance of SIS, CSIS and HOLP here, which can be referred to Table \ref{tab3.1}-\ref{tab3.3}.

\begin{table}[!htb]
\scriptsize
\renewcommand{\arraystretch}{1}
\caption{\footnotesize The proportion of sure screening $\mathbb{P}_s$, the median of minimum required model sizes $\mathbb{M}_s$ and its robust standard deviation RSD (in parentheses) for COLP and FOLP in Example \ref{exm3.1}-\ref{exm3.3}.}\label{tab:sim2}
\begin{tabularx}{\textwidth}{@{}c c Y r Y r Y r Y r@{}}
\toprule
\multirow{3}{*}[-0.5em]{$\mathcal{C}$}
&\multirow{3}{*}[-0.5em]{$R^2$}
&\multicolumn{4}{c}{$(d,n,p)=(200,100,2000)$}
&\multicolumn{4}{c}{$(d,n,p)=(100,200,10000)$}\\
\cmidrule(l){3-6} \cmidrule(l){7-10}
& &\multicolumn{2}{c}{COLP}
&\multicolumn{2}{c}{FOLP}
&\multicolumn{2}{c}{COLP}
&\multicolumn{2}{c}{FOLP} \\
\cmidrule(l){3-4} \cmidrule(l){5-6} \cmidrule(l){7-8} \cmidrule(l){9-10}
& & $\mathbb{P}_s$ & $\mathbb{M}_s$(RSD)
& $\mathbb{P}_s$ & $\mathbb{M}_s$(RSD)
& $\mathbb{P}_s$ & $\mathbb{M}_s$(RSD)
& $\mathbb{P}_s$ & $\mathbb{M}_s$(RSD) \\
\midrule
\multicolumn{10}{c}{Example \ref{exm3.1}}\\
\midrule
$\{1\}$& $60\%$& 0.02 & 397 (558)& 0.00 & 1412 (503)& 0.10 & 412 (691)& 0.05 & 6056 (3495)\\
& $90\%$& 0.85 & 5 (6)& 0.90 & 3 (0)& 1.00 & 3 (0)& 1.00 & 3 (0)\\
\midrule
$\{2\}$& $60\%$& 0.04 & 626 (771)& 0.01 & 1222 (606)& 0.02 & 1200 (2141)& 0.10 & 4406 (4092)\\
& $90\%$& 0.04 & 439 (549)& 0.92 & 3 (0)& 0.19 & 280 (761)& 1.00 & 3 (0)\\
\midrule
$\{3, 4\}$& $60\%$& 0.12 & 238 (473)& 0.16 & 707 (798)& 0.25 & 442 (1027)& 0.36 & 1909 (3787)\\
& $90\%$& 0.29 & 103 (251)& 0.99 & 2 (0)& 0.50 & 38 (146)& 1.00 & 2 (0)\\
\midrule
\multicolumn{10}{c}{Example \ref{exm3.2}}\\
\midrule
$\{1\}$& $60\%$& 0.00 & 924 (779)& 0.00 & 1604 (414)& 0.01 & 1406 (2442)& 0.01 & 7177 (2522)\\
& $90\%$& 0.21 & 83 (151)& 0.40 & 329 (876)& 0.59 & 18 (43)& 0.92 & 4 (0)\\
\midrule
$\{2\}$& $60\%$& 0.01 & 956 (750)& 0.01 & 1474 (512)& 0.02 & 1277 (2536)& 0.03 & 6033 (3389)\\
& $90\%$& 0.05 & 266 (604)& 0.56 & 5 (535)& 0.10 & 234 (705)& 0.96 & 4 (0)\\
\midrule
$\{3, 4\}$& $60\%$& 0.04 & 399 (653)& 0.06 & 1030 (860)& 0.09 & 646 (1994)& 0.16 & 2942 (3847)\\
& $90\%$& 0.27 & 98 (240)& 0.78 & 3 (2)& 0.54 & 30 (184)& 0.99 & 3 (0)\\
\midrule
\multicolumn{10}{c}{Example \ref{exm3.3}}\\
\midrule
$\{1\}$& $60\%$& 0.01 & 581 (576)& 0.00 & 1576 (407)& 0.04 & 612 (885)& 0.02 & 5936 (3368)\\
& $90\%$& 0.47 & 24 (44)& 0.91 & 5 (0)& 0.86 & 6 (8)& 1.00 & 5 (0)\\
\midrule
$\{2\}$& $60\%$& 0.01 & 734 (720)& 0.00 & 1499 (560)& 0.03 & 1424 (1780)& 0.07 & 6066 (3503)\\
& $90\%$& 0.03 & 232 (365)& 0.98 & 5 (1)& 0.22 & 156 (297)& 1.00 & 5 (0)\\
\midrule
$\{3, 4\}$& $60\%$& 0.03 & 637 (765)& 0.06 & 1135 (687)& 0.09 & 1080 (1653)& 0.14 & 5254 (4728)\\
& $90\%$& 0.17 & 214 (397)& 0.94 & 4 (0)& 0.45 & 67 (363)& 1.00 & 4 (0)\\
\bottomrule
\end{tabularx}
\end{table}

We present the conditional screening performance of COLP and FOLP in Example \ref{exm3.1}-\ref{exm3.3} in Table \ref{tab:sim2}. Even though COLP achieves the best overall performance in Simulation study \RNum{1}, we can still see a dramatic increase in the screening accuracy of FOLP when $R^2=90\%$ as shown in Table \ref{tab:sim2}. Such improvement is more significant in the $\mathcal{C}=\{2\}$ and $\mathcal{C}=\{3, 4\}$ cases where FOLP could further eliminate the negative influence of $\beta_1$ as expected when $x_1$ is not included in the prior information.

When $R^2=60\%$, there are a few occasions, such as Example \ref{exm3.1} and \ref{exm3.3} with $\mathcal{C}=\{1\}$, where FOLP is slightly outperformed by COLP. Such results are not surprising since COLP benefits the most from the prior information $\mathcal{C}=\{1\}$ by eliminating the influence of the large coefficient of $x_1$, whereas the proportion of sure screening for FOLP is impaired by considerable errors accumulated in the iterative procedure. Furthermore, large random errors also account for the unstable performance of FOLP in certain scenarios. As shown in Table \ref{tab:sim2}, even though FOLP could achieve the highest proportions of sure screening in many cases when $R^2=60\%$, they are always accompanied by large $\mathbb{M}_s$ and RSD, indicating that it has to select more predictors to achieve the sure screening in many simulations. Nevertheless, the small insufficiency of FOLP can be neglected since the other four screening techniques also struggle to identify all the remaining active predictors and there is no much difference between their screening performance in those challenging situations. Moreover, as illustrated in the following section, the post-screening performance of FOLP remains competitive even when random errors are large.

From the simulation results, we see that FOLP slightly sacrifices its screening accuracy in some challenging cases and achieves dramatic performance improvement in most situations. As introduced in Simulation study \RNum{1}, the three examples are particularly designed against SIS, CSIS and HOLP to demonstrate the attractiveness of COLP. To draw a safer conclusion, we will further examine the effectiveness of FOLP in examples that are widely investigated by other statisticians.

\subsection{Simulation study \RNum{3}: conditional screening performance of FOLP}
In this section, we evaluate the conditional screening performance of FOLP in comparison with that of SIS, CSIS, HOLP and COLP in the four examples that are widely investigated by statisticians. In our simulation, we adopt the same parameter setups and evaluating criteria as in Simulation study \RNum{1} with various conditioning sets.

\begin{exm}\label{exm4.1}
In this example, we investigate the linear model with independent predictors studied by \cite{SIS},\cite{FR} and \cite{HOLP}. Recall that the sure screening property of COLP is based on the normality assumption of predictors. Therefore, to evaluate the performance of COLP and FOLP with non-normally distributed predictors, we consider the linear model with independent and exponentially distributed predictors and random error, where $x_j\sim \mathrm{exp}(1)-1$ for $1\leq j\leq p$ and $\epsilon\sim \mathrm{exp}(\lambda)-1/\lambda$ with $\lambda$ adjusted to achieve predetermined signal ratio $R^2$. The coefficients are designed as
\[\beta_j=(-1)^{U_j}[|Z_j|+4\log(n)/n],\quad\text{for}\quad 1\leq j\leq 8,\]
where $U_j$ follows a Bernoulli distribution with ${P}(U_j=1)=0.4$, $Z_j$ is independent of $U_j$ from the standard normal distribution and $\beta_j=0$ for $j>8$.
\end{exm}

\begin{exm}\label{exm4.2}
We borrow the model from \cite{csis} as
\[y=3x_1+3x_2+3x_3+3x_4+3x_5-7.5x_6+\sum_{j=7}^p\beta_jx_j+\epsilon,\]
where all predictors follow the standard normal distribution with equal correlation $0.5$. Such a correlation structure is named as compound symmetry and was investigated in various literatures \citep{SIS,FR,HOLP}. The random error $\epsilon$ follows a centered normal distribution and the coefficients are determined as $\beta_j=0$ for $j\geq 7$.
\end{exm}

\begin{exm}\label{exm4.3}
We consider the linear model with predictors having the autoregressive correlation structure \citep{LASSO,FR,HOLP}, where all predictors follow the standard normal distribution with covariance $\mathrm{cov}(x_i,x_j)=0.5^{|i-j|}$. The random error $\epsilon$ follows a centered normal distribution and coefficients are chosen as
\[\beta_1=3,\,\,\beta_4=-2,\,\,\beta_7=1.5,\,\,\beta_{10}=-4,\,\,\beta_{13}=2,\]
and the remaining ones are set to zero.
\end{exm}

\begin{exm}\label{exm4.4}
In this example, we consider a challenging case studied by \cite{FR} and \cite{HOLP}. In the linear model, we generate the predictors as $x_j=(z_j+w_j)/\sqrt{2}$ for $1\leq j\leq 5$ and $x_j=(z_j+\sum_{i=1}^5 w_i)/2$ for $6\leq j\leq p$, where $z_j$ and $w_j$ are independent and follow the standard normal distribution. The random error $\epsilon$ follows a centered normal distribution and coefficients are determined as $\beta_j=2j$ for $1\leq j\leq 5$ and $\beta_j=0$ for $6\leq j\leq p$.
\end{exm}

\begin{table}[!htb]
\scriptsize
\renewcommand{\arraystretch}{1}
\caption{\footnotesize The proportion of sure screening $\mathbb{P}_s$, the median of minimum required model sizes $\mathbb{M}_s$ and its robust standard deviation RSD (in parentheses) for screening methods in Example \ref{exm4.1}.}\label{tab4.1}
\begin{tabularx}{\textwidth}{@{}c c Y r Y r Y r Y r Y r@{}}
\toprule
\multirow{2}{*}[-0.4em]{$\mathcal{C}$}
&\multirow{2}{*}[-0.4em]{$R^2$}
&\multicolumn{2}{c}{SIS}
&\multicolumn{2}{c}{CSIS}
&\multicolumn{2}{c}{HOLP}
&\multicolumn{2}{c}{COLP}
&\multicolumn{2}{c}{FOLP} \\
\cmidrule(l){3-4} \cmidrule(l){5-6} \cmidrule(l){7-8} \cmidrule(l){9-10} \cmidrule(l){11-12}
& & $\mathbb{P}_s$ & $\mathbb{M}_s$(RSD)
& $\mathbb{P}_s$ & $\mathbb{M}_s$(RSD)
& $\mathbb{P}_s$ & $\mathbb{M}_s$(RSD)
& $\mathbb{P}_s$ & $\mathbb{M}_s$(RSD)
& $\mathbb{P}_s$ & $\mathbb{M}_s$(RSD) \\
\midrule
\multicolumn{12}{c}{$(d,n,p)=(200,100,2000)$}\\
\midrule
$\{1\}$& $60\%$& 0.00 & 537 (543)& 0.00 & 507 (547)& 0.00 & 574 (538)& 0.01 & 489 (551)& 0.04 & 1466 (604)\\
& $90\%$& 0.06 & 172 (282)& 0.08 & 130 (216)& 0.07 & 136 (285)& 0.10 & 114 (215)& 0.99 & 7 (0)\\
\midrule
$\{2, 3\}$& $60\%$& 0.01 & 502 (548)& 0.03 & 332 (390)& 0.01 & 504 (465)& 0.02 & 338 (399)& 0.05 & 1412 (635)\\
& $90\%$& 0.10 & 136 (234)& 0.26 & 68 (138)& 0.10 & 132 (214)& 0.27 & 67 (111)& 0.99 & 6 (0)\\
\midrule
$\{4, 5, 6\}$& $60\%$& 0.06 & 236 (345)& 0.10 & 138 (218)& 0.05 & 247 (377)& 0.09 & 153 (235)& 0.23 & 998 (1069)\\
& $90\%$& 0.21 & 68 (162)& 0.46 & 26 (69)& 0.23 & 64 (118)& 0.51 & 20 (54)& 0.99 & 5 (0)\\
\midrule
\multicolumn{12}{c}{$(d,n,p)=(100,200,10000)$}\\
\midrule
$\{1\}$& $60\%$& 0.10 & 438 (851)& 0.15 & 326 (835)& 0.09 & 372 (887)& 0.15 & 286 (849)& 0.38 & 1045 (3692)\\
& $90\%$& 0.44 & 44 (112)& 0.54 & 32 (87)& 0.49 & 38 (113)& 0.58 & 29 (73)& 1.00 & 7 (0)\\
\midrule
$\{2, 3\}$& $60\%$& 0.10 & 405 (615)& 0.29 & 167 (307)& 0.10 & 401 (627)& 0.28 & 143 (335)& 0.45 & 350 (3334)\\
& $90\%$& 0.43 & 52 (97)& 0.69 & 17 (28)& 0.48 & 40 (97)& 0.69 & 16 (30)& 1.00 & 6 (0)\\
\midrule
$\{4, 5, 6\}$& $60\%$& 0.37 & 63 (206)& 0.52 & 30 (95)& 0.41 & 56 (182)& 0.51 & 33 (114)& 0.79 & 5 (5)\\
& $90\%$& 0.68 & 16 (52)& 0.89 & 7 (7)& 0.68 & 13 (42)& 0.91 & 7 (8)& 1.00 & 5 (0)\\
\bottomrule
\end{tabularx}
\end{table}

\begin{table}[!htb]
\scriptsize
\renewcommand{\arraystretch}{1}
\caption{\footnotesize The proportion of sure screening $\mathbb{P}_s$, the median of minimum required model sizes $\mathbb{M}_s$ and its robust standard deviation RSD (in parentheses) for screening methods in Example \ref{exm4.2}.}\label{tab4.2}
\begin{tabularx}{\textwidth}{@{}c c Y r Y r Y r Y r Y r@{}}
\toprule
\multirow{2}{*}[-0.4em]{$\mathcal{C}$}
&\multirow{2}{*}[-0.4em]{$R^2$}
&\multicolumn{2}{c}{SIS}
&\multicolumn{2}{c}{CSIS}
&\multicolumn{2}{c}{HOLP}
&\multicolumn{2}{c}{COLP}
&\multicolumn{2}{c}{FOLP} \\
\cmidrule(l){3-4} \cmidrule(l){5-6} \cmidrule(l){7-8} \cmidrule(l){9-10} \cmidrule(l){11-12}
& & $\mathbb{P}_s$ & $\mathbb{M}_s$(RSD)
& $\mathbb{P}_s$ & $\mathbb{M}_s$(RSD)
& $\mathbb{P}_s$ & $\mathbb{M}_s$(RSD)
& $\mathbb{P}_s$ & $\mathbb{M}_s$(RSD)
& $\mathbb{P}_s$ & $\mathbb{M}_s$(RSD) \\
\midrule
\multicolumn{12}{c}{$(d,n,p)=(200,100,2000)$}\\
\midrule
$\{1\}$& $60\%$& 0.00 & 1999 (8)& 0.01 & 386 (437)& 0.00 & 497 (539)& 0.00 & 474 (528)& 0.01 & 1482 (542)\\
& $90\%$& 0.00 & 1999 (0)& 0.10 & 111 (221)& 0.09 & 112 (212)& 0.14 & 98 (183)& 0.97 & 5 (0)\\
\midrule
$\{6\}$& $60\%$& 0.01 & 316 (333)& 0.01 & 310 (348)& 0.00 & 586 (590)& 0.01 & 319 (411)& 0.01 & 1536 (502)\\
& $90\%$& 0.07 & 115 (171)& 0.13 & 84 (128)& 0.04 & 156 (298)& 0.64 & 12 (19)& 0.97 & 5 (0)\\
\midrule
$\{2, 3\}$& $60\%$& 0.00 & 1998 (8)& 0.07 & 257 (417)& 0.04 & 330 (426)& 0.05 & 282 (394)& 0.05 & 1298 (738)\\
& $90\%$& 0.00 & 1998 (0)& 0.30 & 46 (101)& 0.17 & 84 (160)& 0.27 & 62 (98)& 0.99 & 4 (0)\\
\midrule
$\{3, 4, 5\}$& $60\%$& 0.00 & 1997 (8)& 0.04 & 500 (657)& 0.10 & 188 (368)& 0.14 & 188 (339)& 0.10 & 985 (831)\\
& $90\%$& 0.01 & 1997 (0)& 0.25 & 131 (388)& 0.38 & 42 (111)& 0.48 & 22 (74)& 0.99 & 3 (0)\\
\midrule
\multicolumn{12}{c}{$(d,n,p)=(100,200,10000)$}\\
\midrule
$\{1\}$& $60\%$& 0.00 & 9999 (0)& 0.06 & 462 (781)& 0.03 & 462 (734)& 0.05 & 463 (715)& 0.09 & 4759 (3524)\\
& $90\%$& 0.00 & 9999 (0)& 0.52 & 36 (84)& 0.56 & 32 (76)& 0.62 & 22 (65)& 1.00 & 5 (0)\\
\midrule
$\{6\}$& $60\%$& 0.01 & 418 (446)& 0.09 & 330 (466)& 0.01 & 588 (870)& 0.16 & 164 (306)& 0.04 & 6466 (4266)\\
& $90\%$& 0.44 & 52 (82)& 0.64 & 23 (49)& 0.44 & 53 (112)& 0.99 & 5 (1)& 1.00 & 5 (0)\\
\midrule
$\{2, 3\}$& $60\%$& 0.00 & 9998 (0)& 0.24 & 160 (310)& 0.13 & 286 (455)& 0.19 & 218 (450)& 0.32 & 2805 (4812)\\
& $90\%$& 0.00 & 9998 (0)& 0.82 & 10 (14)& 0.63 & 24 (59)& 0.76 & 10 (23)& 1.00 & 4 (0)\\
\midrule
$\{3, 4, 5\}$& $60\%$& 0.00 & 9997 (0)& 0.11 & 812 (2230)& 0.19 & 193 (527)& 0.28 & 112 (288)& 0.34 & 2246 (5005)\\
& $90\%$& 0.00 & 9997 (0)& 0.37 & 68 (296)& 0.68 & 8 (55)& 0.81 & 5 (14)& 1.00 & 3 (0)\\
\bottomrule
\end{tabularx}
\end{table}

\begin{table}[!htb]
\scriptsize
\renewcommand{\arraystretch}{1}
\caption{\footnotesize The proportion of sure screening $\mathbb{P}_s$, the median of minimum required model sizes $\mathbb{M}_s$ and its robust standard deviation RSD (in parentheses) for screening methods in Example \ref{exm4.3}.}\label{tab4.3}
\begin{tabularx}{\textwidth}{@{}c c Y r Y r Y r Y r Y r@{}}
\toprule
\multirow{2}{*}[-0.4em]{$\mathcal{C}$}
&\multirow{2}{*}[-0.4em]{$R^2$}
&\multicolumn{2}{c}{SIS}
&\multicolumn{2}{c}{CSIS}
&\multicolumn{2}{c}{HOLP}
&\multicolumn{2}{c}{COLP}
&\multicolumn{2}{c}{FOLP} \\
\cmidrule(l){3-4} \cmidrule(l){5-6} \cmidrule(l){7-8} \cmidrule(l){9-10} \cmidrule(l){11-12}
& & $\mathbb{P}_s$ & $\mathbb{M}_s$(RSD)
& $\mathbb{P}_s$ & $\mathbb{M}_s$(RSD)
& $\mathbb{P}_s$ & $\mathbb{M}_s$(RSD)
& $\mathbb{P}_s$ & $\mathbb{M}_s$(RSD)
& $\mathbb{P}_s$ & $\mathbb{M}_s$(RSD) \\
\midrule
\multicolumn{12}{c}{$(d,n,p)=(200,100,2000)$}\\
\midrule
$\{1\}$& $60\%$& 0.01 & 740 (726)& 0.02 & 628 (692)& 0.01 & 677 (713)& 0.01 & 516 (678)& 0.23 & 934 (1035)\\
& $90\%$& 0.01 & 446 (628)& 0.07 & 250 (534)& 0.03 & 359 (487)& 0.09 & 209 (385)& 0.99 & 4 (0)\\
\midrule
$\{7\}$& $60\%$& 0.04 & 260 (493)& 0.09 & 197 (415)& 0.09 & 232 (418)& 0.12 & 179 (351)& 0.54 & 6 (671)\\
& $90\%$& 0.21 & 79 (138)& 0.23 & 62 (98)& 0.26 & 56 (96)& 0.32 & 44 (85)& 1.00 & 4 (0)\\
\midrule
$\{4, 7\}$& $60\%$& 0.24 & 76 (196)& 0.30 & 59 (144)& 0.31 & 60 (183)& 0.35 & 44 (130)& 0.73 & 3 (40)\\
& $90\%$& 0.49 & 22 (57)& 0.52 & 19 (42)& 0.56 & 14 (42)& 0.60 & 14 (26)& 1.00 & 3 (0)\\
\midrule
$\{ 10, 13 \}$& $60\%$& 0.01 & 623 (667)& 0.16 & 144 (280)& 0.01 & 552 (673)& 0.12 & 161 (270)& 0.28 & 712 (1035)\\
& $90\%$& 0.04 & 420 (599)& 0.56 & 16 (52)& 0.07 & 326 (443)& 0.61 & 13 (43)& 1.00 & 3 (0)\\
\midrule
\multicolumn{12}{c}{$(d,n,p)=(100,200,10000)$}\\
\midrule
$\{ 1 \}$& $60\%$& 0.00 & 1724 (2666)& 0.02 & 1471 (3182)& 0.04 & 1751 (2081)& 0.02 & 1366 (2825)& 0.56 & 6 (2491)\\
& $90\%$& 0.12 & 496 (1609)& 0.20 & 234 (1001)& 0.13 & 426 (1403)& 0.26 & 247 (665)& 1.00 & 4 (0)\\
\midrule
$\{ 7 \}$& $60\%$& 0.27 & 130 (483)& 0.34 & 82 (378)& 0.32 & 118 (577)& 0.34 & 90 (429)& 0.94 & 4 (0)\\
& $90\%$& 0.59 & 19 (90)& 0.67 & 16 (48)& 0.64 & 16 (76)& 0.69 & 12 (44)& 1.00 & 4 (0)\\
\midrule
$\{ 4, 7 \}$& $60\%$& 0.54 & 27 (94)& 0.59 & 16 (106)& 0.59 & 24 (81)& 0.62 & 16 (97)& 0.97 & 3 (0)\\
& $90\%$& 0.80 & 7 (14)& 0.85 & 6 (12)& 0.84 & 6 (10)& 0.86 & 5 (8)& 1.00 & 3 (0)\\
\midrule
$\{ 10, 13 \}$& $60\%$& 0.01 & 1723 (2755)& 0.31 & 106 (327)& 0.05 & 1750 (2133)& 0.35 & 138 (304)& 0.58 & 4 (2152)\\
& $90\%$& 0.14 & 408 (1624)& 0.94 & 4 (2)& 0.14 & 372 (1422)& 0.94 & 4 (2)& 1.00 & 3 (0)\\
\bottomrule
\end{tabularx}
\end{table}

\begin{table}[!htb]
\scriptsize
\renewcommand{\arraystretch}{1}
\caption{\footnotesize The proportion of sure screening $\mathbb{P}_s$, the median of minimum required model sizes $\mathbb{M}_s$ and its robust standard deviation RSD (in parentheses) for screening methods in Example \ref{exm4.4}.}\label{tab4.4}
\begin{tabularx}{\textwidth}{@{}c c Y r Y r Y r Y r Y r@{}}
\toprule
\multirow{2}{*}[-0.4em]{$\mathcal{C}$}
&\multirow{2}{*}[-0.4em]{$R^2$}
&\multicolumn{2}{c}{SIS}
&\multicolumn{2}{c}{CSIS}
&\multicolumn{2}{c}{HOLP}
&\multicolumn{2}{c}{COLP}
&\multicolumn{2}{c}{FOLP} \\
\cmidrule(l){3-4} \cmidrule(l){5-6} \cmidrule(l){7-8} \cmidrule(l){9-10} \cmidrule(l){11-12}
& & $\mathbb{P}_s$ & $\mathbb{M}_s$(RSD)
& $\mathbb{P}_s$ & $\mathbb{M}_s$(RSD)
& $\mathbb{P}_s$ & $\mathbb{M}_s$(RSD)
& $\mathbb{P}_s$ & $\mathbb{M}_s$(RSD)
& $\mathbb{P}_s$ & $\mathbb{M}_s$(RSD) \\
\midrule
\multicolumn{12}{c}{$(d,n,p)=(200,100,2000)$}\\
\midrule
$\{ 1 \}$& $60\%$& 0.00 & 1999 (0)& 0.00 & 1999 (0)& 0.22 & 218 (677)& 0.20 & 246 (699)& 0.83 & 4 (1)\\
& $90\%$& 0.00 & 1999 (0)& 0.00 & 1999 (0)& 0.56 & 14 (192)& 0.54 & 14 (229)& 1.00 & 4 (0)\\
\midrule
$\{ 5 \}$& $60\%$& 0.00 & 1999 (0)& 0.00 & 1999 (0)& 0.04 & 913 (807)& 0.09 & 782 (818)& 0.45 & 36 (554)\\
& $90\%$& 0.00 & 1999 (0)& 0.00 & 1999 (0)& 0.10 & 592 (875)& 0.18 & 382 (778)& 0.98 & 4 (0)\\
\midrule
$\{ 2, 3 \}$& $60\%$& 0.00 & 1998 (0)& 0.00 & 1998 (0)& 0.19 & 337 (731)& 0.15 & 398 (711)& 0.51 & 20 (225)\\
& $90\%$& 0.00 & 1998 (0)& 0.00 & 1998 (0)& 0.18 & 376 (779)& 0.17 & 451 (845)& 0.97 & 3 (0)\\
\midrule
\multicolumn{12}{c}{$(d,n,p)=(100,200,10000)$}\\
\midrule
$\{ 1 \}$& $60\%$& 0.00 & 9999 (0)& 0.00 & 9999 (0)& 0.33 & 278 (1989)& 0.28 & 344 (2449)& 0.98 & 4 (0)\\
& $90\%$& 0.00 & 9999 (0)& 0.00 & 9999 (0)& 0.51 & 32 (525)& 0.50 & 39 (593)& 1.00 & 4 (0)\\
\midrule
$\{ 5 \}$& $60\%$& 0.00 & 9999 (0)& 0.00 & 9999 (0)& 0.14 & 1966 (4441)& 0.10 & 2137 (3626)& 0.60 & 8 (575)\\
& $90\%$& 0.00 & 9999 (0)& 0.00 & 9999 (0)& 0.20 & 844 (2640)& 0.17 & 822 (2685)& 1.00 & 4 (0)\\
\midrule
$\{ 2, 3 \}$& $60\%$& 0.00 & 9998 (0)& 0.00 & 9998 (0)& 0.32 & 439 (1484)& 0.17 & 1673 (3586)& 0.66 & 4 (551)\\
& $90\%$& 0.00 & 9998 (0)& 0.00 & 9998 (0)& 0.38 & 190 (1461)& 0.20 & 1048 (3104)& 1.00 & 3 (0)\\
\bottomrule
\end{tabularx}
\end{table}

We summarize the simulation results in Table \ref{tab4.1}-\ref{tab4.4}, from which we can draw a similar conclusion to that obtained in Simulation study \RNum{2}, that is, FOLP achieves the overall highest screening accuracy and COLP has the second best performance in most situations. When $R^2=90\%$, FOLP could identify all the remaining active predictors in almost all the simulations and it can also significantly enhance the screening accuracy in many cases with large random errors when $R^2=60\%$.

By considering exponentially distributed predictors and random errors, we in Example \ref{exm4.1} test the effectiveness of COLP and FOLP against the normality assumption that the sure screening property of COLP relies on. We see from Table \ref{tab4.1} that both COLP and FOLP have outstanding performance compared to the other three methods even when the assumption is violated. We also notice that SIS and CSIS could work normally under this setting thanks to the zero correlations among predictors.

In Example \ref{exm4.2}, the active predictor $x_6$ is designed to have zero marginal correlation with the response. Therefore, SIS breaks down as expected and its performance improves when $x_6$ is included in the conditioning set. Compared to HOLP, COLP also benefits significantly from the prior information of $x_6$ through eliminating the adverse impact from its coefficient. Nevertheless, the performance of COLP with other prior information is not that satisfying due to the hidden significant active predictor, whereas FOLP is able to eliminate the influence from coefficients of unidentified active predictors and yields much more accurate screening results regardless of the magnitude of random errors.

In Example \ref{exm4.3}, predictors that are close to each other are highly correlated and those with large distances in position are expected to be mutually independent approximately. Under this setting, CSIS and COLP have similar screening performance, whereas FOLP could further enhance the screening accuracy through the iterative procedure. Finally, in the challenging model introduced in Example \ref{exm4.4}, SIS fails to work since correlations between $x_j$ and $y$ for $1\leq j \leq 5$ are much smaller than those between $x_j$ and $y$ for $j>5$. Moreover, Table \ref{tab4.4} shows that SIS and CSIS have to select all the remaining predictors to cover the true model in all simulations. Meanwhile, HOLP and COLP have much better performance since their sure screening properties do not depend on any covariance assumption. As expected, FOLP achieves the best performance in all the settings, where it could identify all the remaining active predictors in almost all the simulations when $R^2=90\%$. From the extensive simulations, we can safely conclude that FOLP is an extremely competitive conditional variable screening technique compared to other commonly used screening methods.

\subsection{Simulation study \RNum{4}: post-screening performance of FOLP}

As mentioned in Section \ref{subsec:folp}, FOLP could work smoothly without the prior information by using a conditioning set of the predictor obtained by HOLP. We know that variable screening techniques are designed as a preselection step to facilitate further variable selection and parameter estimation procedures. Therefore, in this section, we examine the utility of FOLP in the cases without any prior knowledge by comparing the variable selection performance based on models obtained by FOLP and other screening methods.

In the simulation study, we employ the extended BIC \citep{chen2008} after applying FOLP to select the final model from its solution path by minimizing the objective function
\[\text{EBIC}(\mathcal{S})=\log\left(\frac{\text{RSS}_{\mathcal{S}}}{n}\right)+\frac{|\mathcal{S}|}{n}\cdot(\log n+2\log p),\]
where $\text{RSS}_{\mathcal{S}}$ denotes the RSS for model $\mathcal{S}$ and $|\mathcal{S}|$ denotes its model size. We call this selection procedure as FOLP-EBIC, and compare its performance in Example \ref{exm4.1}-\ref{exm4.4} with that of other two-stage variable selection methods, including SIS-SCAD, ISIS-SCAD, HOLP-LASSO, HOLP-SCAD and FR-EBIC. For SIS and ISIS, we only apply SCAD to conduct the variable selection since it was shown to achieve the best numerical performance by \cite{SIS}. And for FR, we employ EBIC to select the model from its solution path according to \cite{FR}. We also present the selection performance of LASSO and SCAD in the simulation for reference, where all the tuning parameters of them are determined applying EBIC. Lastly, we only consider the $(d,n,p)=(100,200,10000)$ case in Example \ref{exm4.1}-\ref{exm4.4} for simplicity and the selection performance of aforementioned methods is evaluated based on the following criteria.
\begin{itemize}[align=left, leftmargin=*, noitemsep]
\item $\#\text{FNs}$: the average number of active predictors outside the selected models.
\item $\#\text{FPs}$: the average number of inactive predictors included in the selected models.
\item Size: the average model size of the selected models.
\item $\mathbb{P}_s$: the proportion of simulations where the selected models cover the true model.
\item $\mathbb{P}_e$: the proportion of simulations where the selected models equal the true model.
\item $\overline{\text{Err}}$: the average estimation error computed as
\[\overline{\text{Err}}=\frac{1}{d}\sum_{k=1}^d ||\hat{\bm{\beta}}^{(k)}-\bm{\beta}||^2,\]
where $\hat{\bm{\beta}}^{(k)}$ denotes the estimate of $\bm{\beta}$ in the $k$-th simulation.
\item $\hat R^2$: the average out-of-sample $R^2$ computed as
\[\hat R^2=\frac{1}{d}\sum_{k=1}^d\left[1-\frac{||Y^*-X^*\hat{\bm{\beta}}^{(k)}||^2}{\sum_{i=1}^n(Y_i^*-{\bar{Y}}^*)^2}\right]\times 100\%,\]
where $(Y^*,X^*)$ is a set of testing data independent of the $d$ datasets with $Y^*=(Y_1^*,\cdots,Y_n^*)$ and $\bar{Y}^*=\sum_{i=1}^nY_i^*/n$. The out-of-sample $R^2$ measures the effectiveness of $\hat{\bm{\beta}}^{(k)}$ in the out-of-sample forecasting.
\item Time: the average running time of the method in seconds.
\end{itemize}

\begin{table}[!htb]
\footnotesize
\renewcommand{\arraystretch}{1}
\caption{\footnotesize The variable selection performance in Example \ref{exm4.1}-\ref{exm4.4} with $R^2=60\%$.}\label{tab:sim4.1}
\begin{tabularx}{\textwidth}{c l Y c Y c Y c Y c}
\toprule
Example & Method & $\#\text{FNs}$  & $\#\text{FPs}$ & Size & $\mathbb{P}_s$
& $\mathbb{P}_e$ & $\overline{\text{Err}}$ & $\hat R^2$ & Time\\
\midrule
\ref{exm4.1}& LASSO& 7.80& 0.00& 0.20& 0.00& 0.00& 7.28& 1.11& 0.05\\
& SCAD& 7.75& 0.00& 0.25& 0.00& 0.00& 7.26& 1.49& 0.10\\
& SIS-SCAD& 7.81& 0.00& 0.19& 0.00& 0.00& 7.29& 0.99& 0.03\\
& ISIS-SCAD& 7.81& 0.00& 0.19& 0.00& 0.00& 7.29& 0.99& 0.45\\
& HOLP-LASSO& 7.82& 0.00& 0.18& 0.00& 0.00& 7.29& 0.90& 0.09\\
& HOLP-SCAD& 7.81& 0.00& 0.19& 0.00& 0.00& 7.29& 0.99& 0.09\\
& FR-EBIC& 3.89& 0.13& 4.24& 0.05& 0.05& 4.57& 33.43& 0.42\\
& FOLP-EBIC& 3.91& 0.12& 4.21& 0.05& 0.05& 4.58& 33.33& 0.05\\
\midrule
\ref{exm4.2}& LASSO& 5.95& 0.00& 0.05& 0.00& 0.00& 10.07& -0.06& 0.05\\
& SCAD& 5.91& 0.02& 0.11& 0.00& 0.00& 10.04& 0.38& 0.10\\
& SIS-SCAD& 5.98& 0.00& 0.02& 0.00& 0.00& 10.08& -0.31& 0.03\\
& ISIS-SCAD& 5.74& 0.23& 0.49& 0.00& 0.00& 9.80& 2.01& 0.50\\
& HOLP-LASSO& 5.98& 0.00& 0.02& 0.00& 0.00& 10.08& -0.30& 0.09\\
& HOLP-SCAD& 5.90& 0.06& 0.16& 0.00& 0.00& 9.98& 0.51& 0.09\\
& FR-EBIC& 3.32& 1.47& 4.15& 0.00& 0.00& 7.43& 23.68& 0.43\\
& FOLP-EBIC& 3.07& 1.11& 4.04& 0.00& 0.00& 7.13& 25.45& 0.06\\
\midrule
\ref{exm4.3}& LASSO& 3.01& 0.02& 2.01& 0.00& 0.00& 4.46& 27.74& 0.05\\
& SCAD& 2.83& 0.02& 2.19& 0.01& 0.01& 4.22& 30.77& 0.09\\
& SIS-SCAD& 2.82& 0.10& 2.28& 0.00& 0.00& 4.16& 31.40& 0.03\\
& ISIS-SCAD& 2.77& 0.09& 2.32& 0.01& 0.00& 4.17& 31.22& 0.58\\
& HOLP-LASSO& 3.12& 0.01& 1.89& 0.00& 0.00& 4.53& 26.47& 0.09\\
& HOLP-SCAD& 2.82& 0.14& 2.32& 0.00& 0.00& 4.16& 31.43& 0.09\\
& FR-EBIC& 1.39& 0.03& 3.64& 0.13& 0.13& 2.02& 54.97& 0.43\\
& FOLP-EBIC& 1.39& 0.04& 3.65& 0.13& 0.13& 2.02& 54.94& 0.06\\
\midrule
\ref{exm4.4}& LASSO& 4.41& 0.53& 1.12& 0.01& 0.00& 13.42& 8.33& 0.06\\
& SCAD& 1.88& 1.34& 4.46& 0.16& 0.04& 7.36& 43.96& 0.09\\
& SIS-SCAD& 4.37& 1.21& 1.84& 0.00& 0.00& 14.36& 32.47& 0.03\\
& ISIS-SCAD& 4.43& 1.59& 2.16& 0.00& 0.00& 15.42& 26.61& 0.01\\
& HOLP-LASSO& 4.32& 0.65& 1.33& 0.00& 0.00& 13.08& 10.58& 0.09\\
& HOLP-SCAD& 1.89& 1.76& 4.87& 0.03& 0.01& 7.57& 51.33& 0.09\\
& FR-EBIC& 2.45& 0.78& 3.33& 0.00& 0.00& 7.51& 52.75& 0.43\\
& FOLP-EBIC& 1.72& 0.02& 3.30& 0.00& 0.00& 4.37& 56.70& 0.05\\
\bottomrule
\end{tabularx}
\end{table}

\begin{table}[htb]
\footnotesize
\renewcommand{\arraystretch}{1}
\caption{\footnotesize The variable selection performance in Example \ref{exm4.1}-\ref{exm4.4} with $R^2=90\%$.}\label{tab:sim4.2}
\begin{tabularx}{\textwidth}{c l Y c Y c Y c Y c}
\toprule
Example & Method & $\#\text{FNs}$  & $\#\text{FPs}$ & Size & $\mathbb{P}_s$
& $\mathbb{P}_e$ & $\overline{\text{Err}}$ & $\hat R^2$ & Time\\
\midrule
\ref{exm4.1}& LASSO& 0.66& 0.99& 8.33& 0.89& 0.31& 2.94& 71.06& 0.06\\
& SCAD& 0.00& 0.72& 8.72& 1.00& 0.64& 1.02& 88.58& 0.09\\
& SIS-SCAD& 0.95& 0.63& 7.68& 0.35& 0.29& 1.76& 81.09& 0.03\\
& ISIS-SCAD& 0.00& 0.78& 8.78& 1.00& 0.54& 0.82& 89.54& 0.01\\
& HOLP-LASSO& 1.73& 1.01& 7.28& 0.38& 0.08& 3.37& 62.96& 0.09\\
& HOLP-SCAD& 0.86& 0.64& 7.78& 0.38& 0.34& 1.70& 81.71& 0.09\\
& FR-EBIC& 0.00& 0.19& 8.19& 1.00& 0.83& 0.54& 90.53& 0.42\\
& FOLP-EBIC& 0.00& 0.19& 8.19& 1.00& 0.83& 0.54& 90.53& 0.05\\
\midrule
\ref{exm4.2}& LASSO& 4.82& 0.44& 1.62& 0.12& 0.00& 9.35& 12.54& 0.07\\
& SCAD& 0.00& 0.02& 6.02& 1.00& 0.98& 0.72& 89.16& 0.09\\
& SIS-SCAD& 4.50& 0.36& 1.86& 0.00& 0.00& 9.52& 16.65& 0.03\\
& ISIS-SCAD& 0.75& 0.86& 6.11& 0.53& 0.52& 2.34& 83.34& 0.01\\
& HOLP-LASSO& 3.00& 4.84& 7.84& 0.39& 0.00& 6.49& 42.59& 0.09\\
& HOLP-SCAD& 0.73& 0.91& 6.18& 0.43& 0.43& 2.45& 84.07& 0.09\\
& FR-EBIC& 0.00& 0.09& 6.09& 1.00& 0.91& 0.71& 89.19& 0.43\\
& FOLP-EBIC& 0.00& 0.03& 6.03& 1.00& 0.97& 0.69& 89.20& 0.06\\
\midrule
\ref{exm4.3}& LASSO& 0.06& 0.46& 5.40& 0.95& 0.55& 1.97& 79.73& 0.07\\
& SCAD& 0.00& 0.06& 5.06& 1.00& 0.94& 0.49& 88.64& 0.09\\
& SIS-SCAD& 1.27& 0.07& 3.80& 0.11& 0.11& 1.94& 76.96& 0.02\\
& ISIS-SCAD& 0.02& 0.08& 5.06& 0.98& 0.92& 0.45& 88.69& 0.01\\
& HOLP-LASSO& 1.21& 0.41& 4.20& 0.13& 0.07& 2.64& 70.80& 0.09\\
& HOLP-SCAD& 1.21& 0.06& 3.85& 0.13& 0.13& 1.86& 77.67& 0.09\\
& FR-EBIC& 0.00& 0.04& 5.04& 1.00& 0.96& 0.32& 89.13& 0.43\\
& FOLP-EBIC& 0.00& 0.04& 5.04& 1.00& 0.96& 0.32& 89.13& 0.07\\
\midrule
\ref{exm4.4}& LASSO& 0.24& 7.14& 11.90& 0.84& 0.00& 3.17& 85.35& 0.08\\
& SCAD& 0.16& 0.17& 5.01& 0.84& 0.74& 1.22& 89.08& 0.09\\
& SIS-SCAD& 4.34& 1.20& 1.86& 0.00& 0.00& 13.99& 44.06& 0.03\\
& ISIS-SCAD& 4.22& 1.25& 2.03& 0.00& 0.00& 13.66& 45.41& 0.59\\
& HOLP-LASSO& 1.29& 5.87& 9.58& 0.17& 0.00& 4.46& 83.39& 0.09\\
& HOLP-SCAD& 1.27& 0.50& 4.23& 0.14& 0.14& 3.76& 84.24& 0.09\\
& FR-EBIC& 0.37& 0.66& 5.29& 0.63& 0.32& 1.60& 88.94& 0.43\\
& FOLP-EBIC& 0.21& 0.00& 4.79& 0.79& 0.79& 1.09& 89.17& 0.05\\
\bottomrule
\end{tabularx}
\end{table}
The simulation results are presented in Table \ref{tab:sim4.1} and \ref{tab:sim4.2}. We can see that in overall, FOLP-EBIC yields the most accurate screening results in both $R^2=60\%$ and $R^2=90\%$ cases in terms of the proportion of sure screening, the proportion of exact screening, the estimation error and the out-of-sample $R^2$. It also achieves the smallest or close to smallest $\#\text{FNs}$ and $\#\text{FPs}$ in all the simulations. Meanwhile, the screening accuracy of FR-EBIC follows close behind that of FOLP-EBIC in the simulation study. However, the average running time of FR-EBIC is more than eight times of that of FOLP-EBIC in many cases. Such an obvious difference in the computing time can be further enlarged as $p$ increases since FOLP only considers two candidate models in each iteration regardless of the predictor dimension. In fact, even though FOLP-EBIC operates iteratively, its computational efficiency is relatively high compared to the competitors and is only exceeded by that of SIS-SCAD. From the simulation results, we can claim FOLP as a promising screening method according to its outstanding post-screening performance in terms of the selection accuracy and efficiency.

\section{Concluding remarks}
\label{sec:conclusion}
In this paper, we propose a new conditional variable screening method for linear models named as COLP to take advantage of prior information concerning certain active predictors. By eliminating negative influence from coefficients of known active predictors, COLP works competitively in the simulation study and an application to a leukemia dataset. Moreover, we also notice that there is plenty of room for improvement on the screening accuracy of COLP when some significant active predictors are left out of the prior information. To solve this problem, we introduce another screening approach called FOLP to further diminish the adverse effect of hidden active predictors through employing COLP iteratively. Extensive numerical studies show that FOLP has outstanding performance regardless of the availability of prior information. Recall that in the analysis of the leukemia dataset, we initially select the third gene by applying COLP on the linear model. Even though it works out effectively, how to extend similar algorithms to deal with a wider class of models such as generalized linear models is still an interesting topic for the future study.
\section*{Appendix}
\appendix
\section{Preliminaries}
Let $\mathcal{O}(p)$ denote the orthogonal group consisting of all $p\times p$ orthogonal matrices and $V_{n,p}=\{A\in \mathbb{R}^{p\times n}:\,\,A^\top A=I_n\}$ denote the space formed by $n$-frames in $\mathbb{R}^p$. $V_{n,p}$ is called the Stiefel manifold and on the manifold there exists a natural measure $(dX)$ called the Haar measure, which is invariant under both right and left orthogonal transformations \citep{manifold}. By standardization, we can obtain a probability measure as $[dX]=(dX)/V(n,p)$ on the Stiefel manifold with $V(n,p)=2^n\pi^{np/2}/\Gamma_n(p/2)$, where $\Gamma_m(a)=\pi^{m(m-1)/4}\prod_{i=1}^m\Gamma(a-(i-1)/2)$ with $\Gamma$ being the standard gamma function. A random matrix is said to be uniformly distributed on $V_{n,p}$ if its distribution is invariant under both left and right orthogonal transformations, which can be obtained through following decompositions of random matrices.
\begin{defi}[Singular value decomposition, \cite{manifold}, Page 20]
For any $n \times p$ matrix $Z$ with $p\geq n$, there exist $V\in\mathcal{O}(n)$, $U\in V_{n,p}$ and an $n\times n$ diagonal matrix $D$ with non-negative elements, such that
\[Z=VD U^\top .\]
\end{defi}
\begin{defi}[Polar decomposition, \cite{manifold}, Page 19]
For any $p\times n$ matrix $Z$ of rank $n$, we have the unique decomposition
\[Z=H_ZT_Z^{1/2}\quad \text{with} \quad H_Z=Z(Z^\top Z)^{-1/2}\quad\text{and}\quad T_Z=Z^\top Z,\]
where $H_Z\in V_{n,p}$ is called the orientation of $Z$.
\end{defi}

Regarding the distributions of $U$ and $H_Z$ on $V_{n,p}$, we have the following two lemmas, respectively.
\begin{lem}[\cite{SIS}, Lemma 1]\label{fan.svd}
Let $Z$ be an $n\times p$ random matrix with the singular value decomposition $Z=VD U^\top $ and $\bm z^\top _i$ denote the $i$-th row of $Z$, $i=1,2,\cdots,n$. If $\bm z_i$s are independent and their distributions are invariant under right-orthogonal transformations, then $U$ is uniformly distributed on the manifold $V_{n,p}$.
\end{lem}
\begin{lem}[\cite{manifold}, Theorem 2.4.6]\label{macg}
Suppose that a $p\times n$ random matrix $Z$ has the density function of the form
\[f_{Z}(Z)=|\Sigma|^{-n/2}g(Z^\top \Sigma^{-1} Z),\]
where $\Sigma$ is a $p\times p$ positive definite matrix. If the distribution of $Z$ is invariant under the right orthogonal transformation, then its orientation $H_Z$ follows the matrix angular central Gaussian distribution MACG($\Sigma$) on $V_{n,p}$ with the density function
\[f_{H_Z}(H_Z)=|\Sigma|^{-n/2}|H_Z^\top \Sigma^{-1}H_Z|^{-p/2}.\]
\end{lem}

For uniformly distributed matrices on $V_{n,p}$, we also have the following result.

\begin{prop}[\cite{HOLP}, Proposition 2]\label{u.tail}
Let $U$ be uniformly distributed on $V_{n,p}$. Then for any constant $C>0$, there exist constants $\tilde c_1$ and $\tilde c_2$ with $0<\tilde c_1<1<\tilde c_2$, such that
\[P\left(\bm e_1^\top UU^\top \bm e_1<\tilde c_1\cdot\frac{n}{p}\right)<2e^{-Cn}\quad\text{and}\quad P\left(\bm e_1^\top UU^\top \bm e_1>\tilde c_2\cdot\frac{n}{p}\right)< 2e^{-Cn},\]
where $\bm e_i=(0,\cdots,1,0,\cdots,0)^\top $ in this supplement denotes the $i$-th natural base in the corresponding Euclidean space, whose dimensionality is to be understood from the context.
\end{prop}

In addition, the proofs of main theorems also rely on the following results concerning the normal and sub-Gaussian distributions.

\begin{defi}[\cite{manifold}, Page 23]\label{mnorm}
An $n\times p$ random matrix $Z$ is said to follow the rectangular matrix-variate standard normal distribution $N_{n,p}(0;I_n,I_p)$ if it has the density function
\[\varphi^{(n,p)}(Z)=\frac{1}{(2\pi)^{np/2}}\text{\rm{etr}}(-\frac{1}{2}Z^\top Z),\]
where $\text{\rm{etr}}(\cdot)$ denotes $\exp(\text{\rm{trace}}(\cdot))$. Equivalently, the elements of the matrix $Z$ are independent and identically distributed as $N(0,1)$. The $n\times p$ random matrix $W$ is said to follow the normal $N_{n,p}(M;\Sigma_1,\Sigma_2)$ distribution if it can be written as
\[W=\Sigma_1^{\frac{1}{2}}Z\Sigma_2^{\frac{1}{2}}+M,\]
where $Z\sim N_{n,p}(0;I_n,I_p)$, $M$ is an $n\times p$ matrix and $\Sigma_1$ and $\Sigma_2$ are $n\times n$ and $p\times p$ positive definite matrices. The density function of $W$ can be written as
\[\varphi^{(n,p)}(W-M;\Sigma_1,\Sigma_2)=|\Sigma_1|^{-p/2}|\Sigma_2|^{-n/2}\varphi^{(n,p)}[\Sigma_1^{-\frac{1}{2}}(W-M)\Sigma_2^{-\frac{1}{2}}].\]
\end{defi}
\begin{rmk}\label{rmk1}
If matrix $Z$ follows the rectangular matrix-variate standard normal distribution $N_{n,p}(0;I_n,I_p)$ with the singular value decomposition $Z=VD U^\top $, then from Lemma \ref{fan.svd} we know that $U$ is uniformly distributed on $V_{n,p}$.
\end{rmk}

\begin{lem}[\cite{SIS}, Lemma 6 and 7]\label{fan.l1}
Suppose $n\times p$ matrix $Z$ follows the matrix-variate normal distribution $N_{n,p}(0;I_n,I_p)$. Then there exist some $\widetilde c_\lambda>1$ and $\widetilde C_\lambda>0$ such that
\[
P\{\lambda_{\max}(p^{-1}ZZ^\top )>\widetilde c_\lambda\,\,\text{or}\,\,\lambda_{\min}(p^{-1}ZZ^\top )<\widetilde c_\lambda^{-1}\}\leq \exp(-\widetilde C_\lambda n).
\]
\end{lem}
\begin{lem}[\cite{mathsta}, Theorem B.6.5]\label{cdist}
Let $\bm{z}\in\mathbb{R}^p$ follows the multivariate normal distribution $N(\bm{\mu},\Sigma)$ with partitions
\[\bm{z}=\begin{bmatrix}\bm{z}_{1}\\\bm{z}_{2}\end{bmatrix},\quad \bm{\mu}=\begin{bmatrix}\bm{\mu}_{1}\\\bm{\mu}_{2}\end{bmatrix}\quad\text{and}\quad \Sigma=\begin{bmatrix}\Sigma_{11}&\Sigma_{12}\\\Sigma_{21}&\Sigma_{22}\end{bmatrix},\]
where $\bm{z}_{1}$ and $\bm{\mu}_{1}$ are $p_1$-dimensional vectors, $\bm{z}_{2}$ and $\bm{\mu}_{2}$ are $p_2$-dimensional vectors and $\Sigma_{11}$ and $\Sigma_{22}$ are $p_1\times p_1$ and $p_2\times p_2$ matrices with $p_1+p_2=p$. Then, if $\Sigma$ is positive definite, the conditional distribution of $\bm{z}_2$ conditioning on $\bm{z}_1$ can be given by
\[N(\bm{\mu}_2+\Sigma_{21}\Sigma_{11}^{-1}(\bm{z}_1-\bm{\mu}_1),\Sigma_{22\cdot2}),\]
where the covariance matrix $\Sigma_{22\cdot2}=\Sigma_{22}-\Sigma_{21}\Sigma_{11}^{-1}\Sigma_{12}$ is also positive definite.
\end{lem}
\begin{prop}[\cite{ver}, Proposition 5.10]\label{subineq}
Let $\{\xi_i\}_{i=1}^{n}$ be a sequence of i.i.d sub-Gaussian distributed random variables with mean $0$ and a finite variance. Then, there exists a positive constant $C_\xi$ depending on the distribution of $\xi_i$, such that for any vector $\bm a=(a_1,\cdots,a_n)^\top \in\mathbb{R}^n$ with $||\bm a||^2=1$ and every $z\geq 0$,
\[P\left(\left|\sum_{i=1}^n a_i\xi_i\right|\geq z\right)\leq e\cdot\exp\left\{-{C_\xi\cdot z^2}\right\}.\]
\end{prop}

\section{Proof of the main theorems}
Recall that the COLP estimator can be written as
\begin{equation}\label{est1}
\hat{\bm{\beta}}_{\mathcal{D}}=({M}_{\mathcal{C}}{X}_{\mathcal{D}})^+{Y}=({M}_{\mathcal{C}}{X}_{\mathcal{D}})^+{X}{\bm\beta}+({M}_{\mathcal{C}}{X}_{\mathcal{D}})^+{\bm\epsilon},
\end{equation}
where ${M}_{\mathcal{C}}=I_n-{X}_{\mathcal{C}}({X}_{\mathcal{C}}^\top {X}_{\mathcal{C}})^{-1}{X}_{\mathcal{C}}^\top $ and $({M}_{\mathcal{C}}{X}_{\mathcal{D}})^+$ denotes the Moore-Penrose inverse of ${M}_{\mathcal{C}}{X}_{\mathcal{D}}$. Let $C(X_{\mathcal{C}})$ denote the space spanned by the columns of $X_{\mathcal{C}}$ and $C(X_{\mathcal{C}})^{\perp}$ denote its orthogonal complement. With $n_d=n-t_c$, suppose that columns of the matrix $Q_{\mathcal{C}}\in V_{n,n_d}$ form a set of orthonormal basis of the space $C(X_{\mathcal{C}})^{\perp}$. Then we have
\[{X}^\top _{\mathcal{C}}Q_{\mathcal{C}}=\bm{0}\quad\text{and}\quad {M}_{\mathcal{C}}=Q_{\mathcal{C}}Q_{\mathcal{C}}^\top .\]
Moreover, since the matrix $Q_{\mathcal{C}}^\top {X}_{\mathcal{D}}$ is of full row rank, its Moore-Penrose inverse can be written explicitly as
\[(Q_{\mathcal{C}}^\top {X}_{\mathcal{D}})^+={X}^\top _{\mathcal{D}}Q_{\mathcal{C}}(Q_{\mathcal{C}}^\top {X}_{\mathcal{D}}{X}^\top _{\mathcal{D}}Q_{\mathcal{C}})^{-1}.\]
Consequently, the Moore-Penrose inverse $({M}_{\mathcal{C}}{X}_{\mathcal{D}})^+$ can be expressed as
\[({M}_{\mathcal{C}}{X}_{\mathcal{D}})^+=(Q_{\mathcal{C}}^\top {X}_{\mathcal{D}})^+Q_{\mathcal{C}}^\top ={X}^\top _{\mathcal{D}}Q_{\mathcal{C}}(Q_{\mathcal{C}}^\top {X}_{\mathcal{D}}{X}^\top _{\mathcal{D}}Q_{\mathcal{C}})^{-1}Q_{\mathcal{C}}^\top ,\]
where the first equation comes from the facts that $Q_{\mathcal{C}}^+=Q_{\mathcal{C}}^\top $ and $(AB)^+=B^+A^+$ for any matrix $A$ with orthonormal columns. Denoting $W={X}^\top _{\mathcal{D}}Q_{\mathcal{C}}$ and $H_W=W(W^\top W)^{-1/2}$ as its orientation, we can write \eqref{est1} as
\begin{equation}\label{est2}
\hat{\bm\beta}_{\mathcal{D}}=H_WH_W^\top \bm\beta_{\mathcal{D}}+W(W^\top W)^{-1}Q_{\mathcal{C}}^\top \bm\epsilon.
\end{equation}
The main idea of our proofs is to show that $|\hat\beta_i|>|\hat\beta_j|$ with an overwhelming probability for any $i\in\mathcal{T}_\mathcal{D}$ and $j\not\in\mathcal{T}_\mathcal{D}$. To achieve this result, we evaluate the two terms on the right-hand side of \eqref{est2} separately based on the distributions of random matrix $W$ and its orientation $H_W$.
\begin{prop}\label{dist}
Under assumptions (A1) and (A3), the random matrix $W$ follows the rectangular matrix-variate normal distribution $N_{p_d,n_d}(0;\Sigma_{22\cdot2},I_{n_d})$ and its orientation $H_W$ follows the matrix angular central Gaussian distribution MACG$(\Sigma_{22\cdot2})$ with $\Sigma_{22\cdot2}=\Sigma_{22}-\Sigma_{21}\Sigma_{11}^{-1}\Sigma_{12}$.
\end{prop}
\begin{proof}[Proof of Proposition \ref{dist}]
Under assumption (A3), the covariance matrix $\Sigma$ is positive definite since $\text{cond}(\Sigma)$ is bounded from above and then so are $\Sigma_{11}$ and $\Sigma_{22\cdot 2}$. Denote ${X}^\top _{i\mathcal{C}}$ and ${X}^\top _{i\mathcal{D}}$ as the $i$-th row of ${X}_{\mathcal{C}}$ and ${X}_{\mathcal{D}}$, respectively. According to assumption (A1) and Lemma \ref{cdist}, we have
\[{X}_{i\mathcal{D}}|{X}_{i\mathcal{C}}\overset{d}{=}\Sigma^{\frac{1}{2}}_{22\cdot2}{\widetilde Z}_{i}+\Sigma_{21}\Sigma_{11}^{-1}{X}_{i\mathcal{C}},\]
where ${\widetilde Z}_i$s are $p_d$-dimensional i.i.d random vectors from the standard multivariate normal distribution. Consequently, for random matrix ${X}_{\mathcal{D}}$, we have
\[{X}_{\mathcal{D}}|{X}_{\mathcal{C}}\overset{d}{=}{\widetilde Z}\Sigma^{\frac{1}{2}}_{22\cdot2}+{X}_{\mathcal{C}}\Sigma_{11}^{-1}\Sigma_{12},\]
where ${\widetilde Z}$ is an $n\times p_d$ random matrix with i.i.d elements from the standard normal distribution.

Recall that columns of $Q_{\mathcal{C}}$ form a set of orthonormal basis of the space $C({X}_{\mathcal{C}})^{\perp}$ satisfying $Q_{\mathcal{C}}\in V_{n_d,n}$ and ${X}^\top _{\mathcal{C}}{Q}_{\mathcal{C}}=0$. Then, for matrix $W={X}^\top _{\mathcal{D}}Q_{\mathcal{C}}$, we have
\[W|{X}_{\mathcal{C}}\overset{d}{=}\Sigma^{\frac{1}{2}}_{22\cdot2}\widetilde{Z}^\top Q_{\mathcal{C}}\overset{d}{=}\Sigma^{\frac{1}{2}}_{22\cdot2}Z,\]
where $Z$ is a $p_d\times n_d$ random matrix with i.i.d elements from $N(0,1)$ regardless of the choice of $Q_{\mathcal{C}}$. Therefore, $W|{X}_{\mathcal{C}}$ is independent of ${X}_{\mathcal{C}}$ and it follows the rectangular matrix-variate normal distribution $N_{p_d,n_d}(0;\Sigma_{22\cdot2},I_{n_d})$ according to Definition \ref{mnorm}, indicating that $W$ follows the same distribution with the density function
\[f(W)=\frac{1}{(2\pi)^{p_dn_d/2}}\cdot|\Sigma_{22\cdot2}|^{-n_d/2}\rm{etr}(-\frac{1}{2}W^\top \Sigma^{-1}_{22\cdot2}W).\]
Furthermore, notice that $f(W)=f(WQ)$ for any $Q\in\mathcal{O}(n_d)$. Then, from Lemma \ref{macg}, we obtain that $H_W$ follows the MACG($\Sigma_{22\cdot2}$) distribution on $V_{n_d, p_d}$.
\end{proof}
For eigenvalues of $\Sigma$ and $\Sigma_{22\cdot2}$, we have the following result.
\begin{prop}\label{eigen}
Under assumptions (A1) and (A3), we have
\[1/\mathrm{cond}(\Sigma)\leq\lambda_{\text{min}}(\Sigma)\leq \lambda_{\text{min}}(\Sigma_{22\cdot2})\leq \lambda_{\text{max}}(\Sigma_{22\cdot2})\leq \lambda_{\text{max}}(\Sigma)\leq \mathrm{cond}(\Sigma).\]
\end{prop}
\begin{proof}[Proof of Proposition \ref{eigen}]
According to the blockwise inverse formula \citep{binv}, we have
\[\Sigma^{-1}=\begin{bmatrix}\Sigma_{11}&\Sigma_{12}\\\Sigma_{21}&\Sigma_{22}\end{bmatrix}^{-1}=\begin{bmatrix}\Sigma_{11}^{-1}+\Sigma_{11}^{-1}\Sigma_{12}\Sigma_{22\cdot 2}^{-1}\Sigma_{21}\Sigma_{11}^{-1}&-\Sigma_{11}^{-1}\Sigma_{12}\Sigma_{22\cdot 2}^{-1}\\-\Sigma_{22\cdot 2}^{-1}\Sigma_{21}\Sigma_{11}^{-1}&\Sigma_{22\cdot 2}^{-1}\end{bmatrix}.\]
Consequently, we obtain that
\[\lambda_{\text{min}}(\Sigma^{-1})\leq \lambda_{\text{min}}(\Sigma_{22\cdot2}^{-1})\leq \lambda_{\text{max}}(\Sigma_{22\cdot2}^{-1})\leq \lambda_{\text{max}}(\Sigma^{-1}).\]
Equivalently, we have
\[\lambda_{\text{min}}(\Sigma)\leq \lambda_{\text{min}}(\Sigma_{22\cdot2})\leq \lambda_{\text{max}}(\Sigma_{22\cdot2})\leq \lambda_{\text{max}}(\Sigma).\]
Furthermore, notice that $\text{trace}(\Sigma)=\sum_{i=1}^{p}\lambda_i=p$, where $\lambda_i$s denote all the eigenvalues of $\Sigma$. It's obvious that $\lambda_{\text{min}}(\Sigma)\leq 1$ and $\lambda_{\text{max}}(\Sigma)\geq 1$ and thus we can obtain the final conclusion.
\end{proof}
In the next, we introduce two results concerning the quantities of diagonal and off-diagonal terms in $H_WH^\top _W$.
\begin{lem}\label{diag}
Under assumptions (A1) and (A3), for any constant $C>0$, there exist positive constants $0<c_1<1<c_2$, such that for any $i\in\{1,\cdots,p_d\}$,
\[P\left(\bm e_i^\top H_WH_W^\top \bm e_i< c_1\frac{n
^{1-\xi_\lambda}}{p}\right)<2e^{-Cn}\quad\text{and}\quad P\left(\bm e_i^\top H_WH_W^\top \bm e_i> c_2\frac{n^{1+\xi_\lambda}}{p}\right)<2e^{-Cn}.\]
\end{lem}
\begin{proof}[Proof of Lemma \ref{diag}]
More generally, we prove the conclusion for any vector $\bm v\in \mathbb{R}^{p_d}$ with $||\bm v||=1$. Recall that
\begin{equation}\label{d1}
\bm v^\top H_WH_W^\top \bm v =\bm v^\top W(W^\top W)^{-1}W^\top \bm v \overset{d}{=} \bm v^\top \Sigma_{22\cdot2}^{1/2}Z(Z^\top \Sigma_{22\cdot2}Z)^{-1}Z^\top \Sigma_{22\cdot2}^{1/2}\bm v,
\end{equation}
where $Z$ follows the normal distribution $N_{p_d,n_d}(0;I_{p_d},I_{n_d})$. Suppose $Z^\top $ has the SVD as $Z^\top =VDU^\top $, where $V\in\mathcal{O}(n_d)$, $U\in V_{n_d,p_d}$ and $D$ is an $n_d\times n_d$ diagonal matrix. Then, \eqref{d1} can be written as
\[\bm v^\top H_WH_W^\top \bm v\overset{d}{=}\bm v^\top \Sigma_{22\cdot2}^{1/2}U(U^\top \Sigma_{22\cdot2}U)^{-1}U^\top \Sigma_{22\cdot2}^{1/2}\bm v,\]
where $U$ is uniformly distributed on $V_{n_d,p_d}$ according to Remark \ref{rmk1}. In addition, the vector $\Sigma_{22\cdot2}^{1/2}\bm v$ can be expressed as
\[\Sigma_{22\cdot2}^{1/2}\bm v=||\Sigma_{22\cdot2}^{1/2}\bm v||\cdot Q\bm e_1,\]
where $Q$ is some $p_d\times p_d$ orthogonal matrix. Consequently, we have
\begin{align}
\bm v^\top H_WH_W^\top \bm v&\overset{d}{=}||\Sigma_{22\cdot2}^{1/2}\bm v||^2\cdot \bm e_1^\top  Q^\top U(U^\top \Sigma_{22\cdot2}U)^{-1}U^\top Q\bm e_1\nonumber\\
&\overset{d}{=}||\Sigma_{22\cdot2}^{1/2}\bm v||^2\cdot \bm e_1^\top  \widetilde U(U^\top \Sigma_{22\cdot2}U)^{-1}\widetilde U^\top \bm e_1,\label{d2}
\end{align}
where $\widetilde U=Q^\top U$ is also uniformly distributed on $V_{n_d,p_d}$. For the norm term, we have
\begin{equation}\label{d3}
\lambda_{\rm{min}}(\Sigma_{22\cdot2})\leq ||\Sigma_{22\cdot2}^{1/2}\bm v||^2\leq \lambda_{\rm{max}}(\Sigma_{22\cdot2}).
\end{equation}
Furthermore, we have
\begin{align}
\lambda_{\rm{max}}^{-1}(\Sigma_{22\cdot2})||\widetilde U^\top \bm e_1||^2&\leq \lambda_{\rm{min}}((U^\top \Sigma_{22\cdot2}U)^{-1})||\widetilde U^\top \bm e_1||^2\leq\bm e_1^\top \widetilde U(U^\top \Sigma_{22\cdot2}U)^{-1}\widetilde U^\top \bm e_1\nonumber\\
&\leq \lambda_{\rm{max}}((U^\top \Sigma_{22\cdot2}U)^{-1})||\widetilde U^\top \bm e_1||^2 \leq \lambda_{\rm{min}}^{-1}(\Sigma_{22\cdot2})||\widetilde U^\top \bm e_1||^2.\label{d4}
\end{align}
Combining \eqref{d2}, \eqref{d3} and \eqref{d4}, we have
\[\frac{\lambda_{\rm{min}}(\Sigma_{22\cdot2})}{\lambda_{\rm{max}}(\Sigma_{22\cdot2})}||\widetilde U^\top \bm e_1||^2\leq \bm v^\top H_WH_W^\top \bm v\leq \frac{\lambda_{\rm{max}}(\Sigma_{22\cdot2})}{\lambda_{\rm{min}}(\Sigma_{22\cdot2})}||\widetilde U^\top \bm e_1||^2.\]
Under assumption (A3), we have $\text{cond}(\Sigma_{22\cdot2})\leq\text{cond}(\Sigma)\leq c_\lambda n^{\xi_\lambda}$ by Proposition \ref{eigen}. From Proposition \ref{u.tail}, for any constant $C>0$, there exist constants $\tilde c_1^*$ and $\tilde c_2^*$ with $0<\tilde c_{1}^*<1<\tilde c_{2}^*$, such that
\[P\left(||\widetilde U^\top \bm e_1||^2<\tilde c_1^*\frac{n_d}{p_d}\right)<2e^{-2Cn_d}\quad\text{and}\quad P\left(||\widetilde U^\top \bm e_1||^2>\tilde c_2^*\frac{n_d}{p_d}\right)<2e^{-2Cn_d}.\]
Combining with the fact that $n/2\leq n_d\leq n$ and $p/2\leq p_d\leq p$ from assumption (A3), we obtain
\[P\left(\bm v^\top H_WH_W^\top \bm v<c_1\frac{n^{1-\xi_\lambda}}{p}\right)<2e^{-Cn}\quad\text{and}\quad P\left(\bm v^\top H_WH_W^\top \bm v>c_2\frac{n^{1+\xi_\lambda}}{p}\right)<2e^{-Cn},\]
where $c_1=\tilde c_1^*/4c_\lambda$ and $c_2=4\tilde c_2^*c_\lambda$.
\end{proof}
\begin{lem}\label{l.off}
Suppose $H$ follows the MACG($\Sigma$) distribution on $V_{n,p}$ with $\text{cond}(\Sigma)\leq c^*n^{\tau}$ for some positive constants $c^*$. Then, for any $0<\alpha<0.5$ and $C>0$, there exists some positive constant $\tilde c_3$, such that for any $i,j\in\{1,\cdots,p\}$ with $i\neq j$,
\[P\left(\big| \bm e_i^\top HH^\top  \bm e_j\big|>\frac{\tilde c_3n^{1+\tau-\alpha}}{p\sqrt{\log n}}\right)\leq O\left\{\exp\left(\frac{-Cn^{1-2\alpha}}{{\log n}}\right)\right\}.\]
\end{lem}
\begin{rmk}
The proof of Lemma \ref{l.off} can be referred to the proof of Lemma 5 in \cite{HOLP}.
\end{rmk}
\begin{cor}\label{c.off}
Under assumptions (A1) and (A3), for any $C>0$, there exists some positive constant $c_3$, such that for any $i,j\in\{1,\cdots,p_d\}$ with $i\neq j$,
\[P\left(\big|\bm e_i^\top H_WH_W^\top \bm e_j\big|>\frac{c_3n^{1-0.5\xi_t-\xi_\beta-1.5\xi_\lambda}}{p\sqrt{\log n}}\right)\leq O\left\{\exp\left(\frac{-Cn^{1-\xi_t-2\xi_\beta-5\xi_\lambda}}{{\log n}}\right)\right\},\]
where $\xi_t+2\xi_\beta+5\xi_\lambda<1$ under assumption (A3).
\end{cor}

Based on the results regarding elements of $H_WH_W^\top $, we are able to estimate the first term $H_WH_W^\top \bm \beta_\mathcal{D}$.
\begin{lem}\label{1t}
Under assumptions (A1) and (A3), for any positive constant $C$, there exist constants $c_4,c_5>0$, such that for any $i\in\mathcal{T}_\mathcal{D}$,
\[P\left(\big|\bm  e_i^\top H_WH_W^\top \bm \beta_\mathcal{D}\big|< c_4\cdot\frac{n^{1-\xi_\beta-\xi_\lambda}}{p}\right)\leq O\left\{\exp\left(\frac{-Cn^{1-\xi_t-2\xi_\beta-5\xi_\lambda}}{2\cdot{\log n}}\right)\right\},\]
and for any $j\not\in\mathcal{T}_\mathcal{D}$,
\[P\left(\big|\bm e_j^\top H_WH_W^\top \bm \beta_\mathcal{D}\big|>\frac{c_5}{\sqrt{\log n}}\frac{n^{1-\xi_\beta-\xi_\lambda}}{p}\right)\leq O\left\{\exp\left(\frac{-Cn^{1-\xi_t-2\xi_\beta-5\xi_\lambda}}{2\cdot{\log n}}\right)\right\}.\]
\end{lem}
\begin{proof}[Proof of Lemma \ref{1t}]
Under assumptions (A1) and (A3), we have
$$\text{E}[\text{var}(y|\bm x_\mathcal{C})]=\bm \beta_{\mathcal{D}}^\top \Sigma_{22\cdot 2}\bm \beta_{\mathcal{D}}+\sigma^2\leq C_v,$$
where $C_v$ is some finite constant. Consequently, by Proposition \ref{eigen}, for $i\in \mathcal{T}_D$  we have
\[\big|\beta_{i}\big|^2\leq ||\bm \beta_{\mathcal{D}}||^2\leq C_v/\lambda_{\text{min}}(\Sigma_{22\cdot 2})\leq C_v\cdot\text{cond}(\Sigma)\leq C_vc_\lambda n^{\xi_\lambda}.\]
For any $C>0$ and $j\not\in\mathcal{T}_\mathcal{D}$, according to Corollary \ref{c.off}, with a probability of at least $1-O\left\{n^{\xi_t}\cdot\exp\left({-Cn^{1-\xi_t-2\xi_\beta-5\xi_\lambda}}/{{\log n}}\right)\right\}$, we have
\begin{align}
\big|\bm e_j^\top H_WH_W^\top \bm \beta_\mathcal{D}\big|&\leq \sum_{i\in\mathcal{T}_\mathcal{D}}\big|\bm e_j^\top H_WH_W^\top \bm e_i\big|\cdot \big|\beta_{i}\big|\leq \sqrt{\sum_{i\in\mathcal{T}_\mathcal{D}}\big|\bm e_j^\top H_WH_W^\top \bm e_i\big|^2}\cdot \big|\big|\bm \beta_{\mathcal{D}}\big|\big|\nonumber\\
&\leq \sqrt{c_t n^{\xi_t} \cdot \frac{c_3^{2}\cdot n^{2-\xi_t-2\xi_\beta-3\xi_\lambda}}{p^2\log n}\cdot C_vc_\lambda n^{\xi_\lambda}}= \frac{c_5n^{1-\xi_\beta-\xi_\lambda}}{p\sqrt{\log n}},\nonumber
\end{align}
where $c_5=c_3\sqrt{C_vc_t c_\lambda}$. As a consequence, we obtain
\begin{equation}\label{1t.2}
P\left(\big|\bm e_j^\top H_WH_W^\top \bm \beta_\mathcal{D}\big|>\frac{c_5}{\sqrt{\log n}}\frac{n^{1-\xi_\beta-\xi_\lambda}}{p}\right)\leq O\left\{\exp\left(\frac{-Cn^{1-\xi_t-2\xi_\beta-5\xi_\lambda}}{2\cdot{\log n}}\right)\right\}.
\end{equation}
Moreover, for $i\in\mathcal{T}_\mathcal{D}$, according to Lemma \ref{diag} and Corollary \ref{c.off}, with a probability of at least
\[1-2e^{-Cn}-O\left\{n^{\xi_t}\cdot\exp\left(\frac{-Cn^{1-\xi_t-2\xi_\beta-5\xi_\lambda}}{{\log n}}\right)\right\},\]
we have
\begin{align}
\big|\bm e_i^\top H_WH_W^\top \bm \beta_\mathcal{D}\big|&\geq \big|\bm e_i^\top H_WH_W^\top \bm e_i\big|\cdot \beta_{\text{min}}-\sum_{j\neq i,j\in\mathcal{T}_\mathcal{D}}\big|\bm e_i^\top H_WH_W^\top \bm e_j\big|\cdot \big|\beta_{j}\big|\nonumber\\
&\geq c_1c_\beta\frac{n^{1-\xi_\beta-\xi_\lambda}}{p}-\frac{c_5}{\sqrt{\log n}}\frac{n^{1-\xi_\beta-\xi_\lambda}}{p}\geq c_4\cdot\frac{n^{1-\xi_\beta-\xi_\lambda}}{p}, \label{diageq}
\end{align}
where $c_4=c_1c_\beta/2$. Therefore, we obtain
\begin{equation}\label{1t.1}
P\left(\big|\bm e_i^\top H_WH_W^\top \bm \beta_\mathcal{D}\big|< c_4\cdot\frac{n^{1-\xi_\beta-\xi_\lambda}}{p}\right)\leq O\left\{\exp\left(\frac{-Cn^{1-\xi_t-2\xi_\beta-5\xi_\lambda}}{2\cdot{\log n}}\right)\right\}.
\end{equation}
\end{proof}
For the second term $W(W^\top W)^{-1}Q_{\mathcal{C}}^\top \bm \epsilon$ on the right-hand side of \eqref{est2}, we have the following result.
\begin{lem}\label{2t}
Under assumptions (A1), (A2) and (A3), there exist some positive constants $c_6$ and $C_0$, such that
\[P\left(\big|\bm e_i^\top  W(W^\top W)^{-1}Q_{\mathcal{C}}^\top \bm \epsilon\big|>\frac{c_6\cdot n^{1-\xi_\beta-\xi_\lambda}}{\sqrt{\log n}\cdot p}\right)<O\left\{\exp\left(\frac{-C_0n^{1-2\xi_\beta-4\xi_\lambda}}{{\log n}}\right)\right\}.\]
\end{lem}
\begin{proof}[Proof of Lemma \ref{2t}]
Denote
\[\bm \eta=(\eta_1,\cdots,\eta_{p_d})^\top =W(W^\top W)^{-1}Q_{\mathcal{C}}^\top \bm \epsilon,\]
and
\[a_{ij}=\bm {e}_i^\top W(W^\top W)^{-1}Q_{\mathcal{C}}^\top  \bm e_j/\sqrt{\bm {e}_i^\top W(W^\top W)^{-2}W^\top \bm {e}_i},\]
where the numerator denotes the element of $W(W^\top W)^{-1}Q_{\mathcal{C}}^\top $ in the $i$-th row and $j$-th column and the denominator denotes the norm of the $i$-th row vector in $W(W^\top W)^{-1}Q_{\mathcal{C}}^\top $.
Therefore, we have
\begin{equation}\label{2t.0}
\eta_i=\sqrt{\bm {e}_i^\top W(W^\top W)^{-2}W^\top \bm {e}_i}\cdot \bm{a_i\epsilon},
\end{equation}
where $\bm{a_i}=(a_{i1},\cdots,a_{in})$ with $||\bm{a_i}||=1$. For the scalar term $\sqrt{\bm {e}_i^\top W(W^\top W)^{-2}W^\top \bm {e}_i}$, we have
\begin{align}
\bm {e}_i^\top W(W^\top W)^{-2}W^\top \bm {e}_i&\leq \lambda_{\text{max}}((W^\top W)^{-1})\cdot\bm {e}_i^\top H_WH_W^\top \bm {e}_i\nonumber\\
&=\lambda^{-1}_{\text{min}}(Z^\top \Sigma_{22\cdot 2}Z)\cdot\bm {e}_i^\top H_WH_W^\top \bm {e}_i,\label{2t.1}
\end{align}
where $Z$ follows the normal distribution $N_{p_d,n_d}(0;I_{p_d},I_{n_d})$. According to Lemma \ref{fan.l1}, there exist some $\widetilde C_\lambda>0$ and $\widetilde c_\lambda>1$, such that
\[P\left(\lambda_{\text{min}}(p_d^{-1}Z^\top Z)<\widetilde c_\lambda^{-1}\right)<e^{-\widetilde C_\lambda n_d}.\]
Consequently, together with the fact that
\[
\lambda^{-1}_{\text{min}}(Z^\top \Sigma_{22\cdot 2}Z)\leq\lambda^{-1}_{\text{min}}(\Sigma_{22\cdot 2})\cdot \lambda^{-1}_{\text{min}}(Z^\top Z)\leq \text{cond}(\Sigma)\cdot p_d^{-1}\cdot\lambda^{-1}_{\text{min}}(p_d^{-1}Z^\top Z),
\]
we have
\begin{equation}\label{2t.2}
P\left(\lambda^{-1}_{\text{min}}(Z^\top \Sigma_{22\cdot 2}Z)>2{c_\lambda}{\widetilde c_\lambda}\cdot\frac{n^{\xi_\lambda}}{p}\right)<e^{-\widetilde C_\lambda n/2}.
\end{equation}
Meanwhile, according to Lemma \ref{diag}, for the same $\widetilde C_\lambda$, there exists some positive constant $c_2$ such that
\begin{equation}\label{2t.3}
P\left(\bm e_i^\top H_WH_W^\top \bm e_i>c_2\cdot\frac{n^{1+\xi_\lambda}}{p}\right)<2e^{-\widetilde C_\lambda n}.
\end{equation}
Therefore, combining \eqref{2t.1}, \eqref{2t.2} and \eqref{2t.3}, we have
\begin{equation}\label{2t.4}
P\left(\bm {e}_i^\top W(W^\top W)^{-2}W^\top \bm {e}_i>2{c_2c_\lambda\widetilde c_\lambda }\cdot\frac{n^{1+2\xi_\lambda}}{p^2}\right)<O\left\{\exp\left({-\widetilde C_\lambda n/2}\right)\right\}.
\end{equation}
Furthermore, by Proposition \ref{subineq}, letting $z=\sqrt{n^{1-2\xi_\beta-4\xi_\lambda}/{\log n}}$, there exists some positive constant $C_0$ such that
\begin{equation}\label{2t.5}
P\left(\left|\sum_{j=1}^n a_{ij}\epsilon_j\right|>\sqrt{\frac{n^{1-2\xi_\beta-4\xi_\lambda}}{\log n}}\right)<O\left\{\exp\left(\frac{-C_0 n^{1-2\xi_\beta-4\xi_\lambda}}{{\log n}}\right)\right\}.
\end{equation}
Together with \eqref{2t.4} and \eqref{2t.5}, denoting $c_6=\sigma\sqrt{2c_2c_\lambda\widetilde c_\lambda}$, we obtain that
\[P\left(\big|\eta_i\big|>\frac{c_6\cdot n^{1-\xi_\beta-\xi_\lambda}}{\sqrt{\log n}\cdot p}\right)<O\left\{\exp\left(\frac{-C_0n^{1-2\xi_\beta-4\xi_\lambda}}{{\log n}}\right)\right\}.\]
\end{proof}
\begin{proof}[Proof of Theorem 1]
Recall that for any $i\in\mathcal{D}$, we have the corresponding COLP estimator
\[\hat{ \beta}_{i}=\bm e_i^\top H_WH_W^\top \bm \beta_\mathcal{D}+\bm e_i^\top  W(W^\top W)^{-1}Q_{\mathcal{C}}^\top \bm \epsilon.\]
According to Lemma \ref{1t}, for the same $C_0$ as defined in Lemma \ref{2t}, there exists some constant $c_4>0$ such that
\begin{align*}
P\left(\min_{i\in \mathcal{T}_\mathcal{D}}\big|\bm e_i^\top H_WH_W^\top \bm \beta_\mathcal{D}\big|<c_4\frac{n^{1-\xi_\beta-\xi_\lambda}}{p}\right)&< O\left\{c_t n^{\xi_t}\cdot\exp\left(\frac{-C_0\cdot n^{1-\xi_t-2\xi_\beta-5\xi_\lambda}}{2\cdot{\log n}}\right)\right\}\\
&< O\left\{\exp\left(\frac{-C_0\cdot n^{1-\xi_t-2\xi_\beta-5\xi_\lambda}}{3\cdot{\log n}}\right)\right\}.
\end{align*}
By Lemma \ref{2t}, we also have
\begin{align*}
P\left(\max_{i\in \mathcal{T}_\mathcal{D}}\big|\bm e_i^\top  W(W^\top W)^{-1}Q_{\mathcal{C}}^\top \bm \epsilon\big|>\frac{c_6\cdot n^{1-\xi_\beta-\xi_\lambda}}{\sqrt{\log n}\cdot p}\right)&<O\left\{c_t n^{\xi_t}\cdot\exp\left(\frac{-C_0\cdot n^{1-2\xi_\beta-4\xi_\lambda}}{{\log n}}\right)\right\}\\
&< O\left\{\exp\left(\frac{-C_0\cdot n^{1-2\xi_\beta-4\xi_\lambda}}{{2\cdot\log n}}\right)\right\}.
\end{align*}
Therefore, for any threshold parameter $\gamma_n$ satisfying that
\[\frac{n^{1-\xi_\beta-\xi_\lambda}}{\sqrt{\log n}\cdot p}=o(\gamma_n)\quad\text{and}\quad \gamma_n = o\left(\frac{n^{1-\xi_\beta-\xi_\lambda}}{p}\right),\]
denoting $C=C_0/3$, we obtain that
\begin{align*}
P\left(\min_{i\in \mathcal{T}_\mathcal{D}}\big|\hat\beta_{i}\big|\leq\gamma_n\right)&\leq P\left(\min_{i\in \mathcal{T}_\mathcal{D}}\big|\bm e_i^\top H_WH_W^\top \bm \beta_\mathcal{D}\big|<c_4\frac{n^{1-\xi_\beta-\xi_\lambda}}{p}\right)\\
&+P\left(\max_{i\in \mathcal{T}_\mathcal{D}}\big|\bm e_i^\top  W(W^\top W)^{-1}Q_{\mathcal{C}}^\top \bm \epsilon\big|>\frac{c_6\cdot n^{1-\xi_\beta-\xi_\lambda}}{\sqrt{\log n}\cdot p}\right)\\
&< O\left\{\exp\left(\frac{-C_0\cdot n^{1-\xi_t-2\xi_\beta-5\xi_\lambda}}{3\cdot{\log n}}\right)\right\}\\
&=O\left\{\exp\left(\frac{-C\cdot n^{1-\xi_t-2\xi_\beta-5\xi_\lambda}}{{\log n}}\right)\right\}.
\end{align*}
Consequently, if we determine the final model applying such a threshold parameter, we have
\[P\left(\mathcal{T}_\mathcal{D}\subset \mathcal{S}^{\gamma_n}\right)\geq 1- O\left\{\exp\left(\frac{-C\cdot n^{1-\xi_t-2\xi_\beta-5\xi_\lambda}}{{\log n}}\right)\right\}.\]
\end{proof}
\begin{proof}[Proof of Theorem 2]
Taking the same $C_0$ as defined in Theorem 1, by Lemma \ref{1t}, we know that there exists some $c_5>0$ such that
\[P\left(\max_{j\not\in\mathcal{T}_\mathcal{D}}\big|\bm e_j^\top H_WH_W^\top \bm \beta_\mathcal{D}\big|>\frac{c_5}{\sqrt{\log n}}\frac{n^{1-\xi_\beta-\xi_\lambda}}{p}\right)\leq O\left\{p_d\cdot\exp\left(\frac{-C_0n^{1-\xi_t-2\xi_\beta-5\xi_\lambda}}{2\cdot{\log n}}\right)\right\}.\]
According to Lemma \ref{2t}, we also have
\[P\left(\max_{j\not\in\mathcal{T}_\mathcal{D}}\big|\bm e_j^\top  W(W^\top W)^{-1}Q_{\mathcal{C}}^\top \bm \epsilon\big|>\frac{c_6\cdot n^{1-\xi_\beta-\xi_\lambda}}{\sqrt{\log n}\cdot p}\right)<O\left\{p_d\cdot \exp\left(\frac{-C_0n^{1-2\xi_\beta-4\xi_\lambda}}{{\log n}}\right)\right\}.\]
Let $\gamma_n$ be a threshold parameter follows the same assumption in Theorem 1. Under the assumption that
\[\log p = o\left(\frac{n^{1-\xi_t-2\xi_\beta-5\xi_\lambda}}{{\log n}}\right),\]
we have
\begin{align*}
P\left(\max_{j\not\in \mathcal{T}_\mathcal{D}}\big|\hat\beta_{j}\big|\geq\gamma_n\right)&\leq P\left(\max_{j\not\in\mathcal{T}_\mathcal{D}}\big|\bm e_j^\top H_WH_W^\top \bm \beta_\mathcal{D}\big|>\frac{c_5}{\sqrt{\log n}}\frac{n^{1-\xi_\beta-\xi_\lambda}}{p}\right)\\
&+P\left(\max_{j\not\in\mathcal{T}_\mathcal{D}}\big|\bm e_j^\top  W(W^\top W)^{-1}Q_{\mathcal{C}}^\top \bm \epsilon\big|>\frac{c_6\cdot n^{1-\xi_\beta-\xi_\lambda}}{\sqrt{\log n}\cdot p}\right)\\
&< O\left\{\exp\left(\log p_d-\frac{C_0n^{1-\xi_t-2\xi_\beta-5\xi_\lambda}}{2\cdot{\log n}}\right)\right\}\\
&< O\left\{\exp\left(\frac{-C\cdot n^{1-\xi_t-2\xi_\beta-5\xi_\lambda}}{{\log n}}\right)\right\},
\end{align*}
where $C=C_0/3$ as defined in Theorem 1. Combining with Theorem 1, we have
\[P\left(\max_{j\not\in \mathcal{T}_\mathcal{D}}\big|\hat\beta_{j}\big|<\gamma_n<\min_{i\not\in \mathcal{T}_\mathcal{D}}\big|\hat\beta_{i}\big|\right)\geq 1-O\left\{\exp\left(\frac{-C\cdot n^{1-\xi_t-2\xi_\beta-5\xi_\lambda}}{{\log n}}\right)\right\}.
\]
Therefore, if we choose the model with size $d_n\geq c_t n^{\xi_t} \geq t_c$, we have
\[P\left(\mathcal{T}_\mathcal{D}\subset \mathcal{S}_{d_n}\right)\geq 1- O\left\{\exp\left(\frac{-C\cdot n^{1-\xi_t-2\xi_\beta-5\xi_\lambda}}{{\log n}}\right)\right\}.\]
\end{proof}

\bibliographystyle{abbrvnat}
\bibliography{test}
\end{document}